\documentclass[number]{elsarticle}
\usepackage{amsfonts, amsmath, amstext,amsthm, amssymb,bbm}

\journal{Journal of Algebra}

\usepackage[colorlinks]{hyperref}

\RequirePackage{color}
\definecolor{e-mail}{rgb}{0,.40,.80}
\definecolor{reference}{rgb}{.20,.60,.22}
\definecolor{citation}{rgb}{0,.40,.80}

\theoremstyle{plain}
\newtheorem{theorem}{Theorem}[section]
\newtheorem{lemma}[theorem]{Lemma}
\newtheorem{proposition}[theorem]{Proposition}
\newtheorem{fact}[theorem]{Fact}
\newtheorem{corollary}[theorem]{Corollary}

\theoremstyle{definition}
\newtheorem{definition}[theorem]{Definition}

\newcommand{\Inf}{\text{Inf}}
\newcommand{\set}[2]{\ensuremath{ {\left\{ #1 : #2 \right\}} }}
\DeclareMathOperator{\Deg}{deg}
\DeclareMathOperator{\ord}{ord}
\renewcommand{\deg}[1]{\ensuremath{\text{deg}(#1)}}
\newcommand{\F}{\mathbb{F}}
\newcommand{\Quot}{\mathrm{Quot}}
\newcommand{\Q}{\mathbb{Q}}
\newcommand{\R}{\mathcal{R}}
\newcommand{\Fbar}{\overline{F}}
\newcommand{\TKbar}{\overline{T_K}}
\newcommand{\fbar}{\overline{f}}
\newcommand{\gbar}{\overline{g}}
\newcommand{\gtilde}{\tilde{g}}
\newcommand{\htilde}{\tilde{h}}
\newcommand{\ptilde}{\tilde{p}}
\newcommand{\qtilde}{\tilde{q}}
\newcommand{\utilde}{\tilde{u}}
\DeclareMathOperator{\rank}{rank}

\newcommand{\xtilde}{\tilde{x}}
\newcommand{\ytilde}{\tilde{y}}
\newcommand{\ztilde}{\tilde{z}}
\newcommand{\xvec}{\vec{x}}
\newcommand{\del}{\delta}

\newcommand{\KY}{K\{ Y\}}

\newcommand{\ACF}{\ensuremath{\textbf{ACF}_0}\ }
\newcommand{\DCF}{\ensuremath{\textbf{DCF}_0}\ }

\newcommand{\la}{\langle}
\newcommand{\ra}{\rangle}

\def\diverges{\!\uparrow}
\def\converges{\!\downarrow}

\newcommand{\bfd}{\bf{d}}
\def\bfz{\bf{0}}
\def\Fhat{\ensuremath{\hat{F}}}
\def\Khat{\ensuremath{\hat{K}}}

\usepackage[totalheight=8.2in, totalwidth=5.5in,paperwidth=7in,paperheight=10in]{geometry}

\usepackage{setspace}

\begin{document}
\begin{frontmatter}
\title{Computing Constraint Sets for Differential Fields\fnref{conference}}

\author{Russell Miller\fnref{RM}}
\ead{Russell.Miller@qc.cuny.edu}
\address{
Department of Mathematics, CUNY Queens College,
65-30 Kissena Blvd., Queens, New York  11367, USA\\
Ph.D. Programs in Mathematics \& Computer Science, CUNY Graduate Center,
365 Fifth Avenue, New York, New York  10016, USA}

\author{Alexey Ovchinnikov\fnref{AO}}
\ead{aovchinnikov@qc.cuny.edu}
\address{Department of Mathematics,
CUNY Queens College, 65-30 Kissena Blvd.,
Queens, New York 11367, USA\newline
Ph.D. Program in Mathematics, CUNY Graduate Center,
365 Fifth Avenue, New York, New York  10016, USA}

\author{Dmitry Trushin}
\ead{trushindima@yandex.ru}
\address{Einstein Institute of Mathematics,
Edmond J. Safra Campus, Givat Ram,
The Hebrew University of Jerusalem,
Jerusalem, 91904, Israel}

\fntext[RM]{R.~Miller was
partially supported by the NSF grant DMS-1001306, by the \emph{Infinity Project}
of the Templeton Foundation at the Centre de Recerca Matem\'atica (grant \#13152),
by the Isaac Newton Institute as a Visiting Fellow,
by the European Science Foundation (grant \#4610), and by grants 
\#62632-00 40 and \#63286-00 41 from
the PSC-CUNY Research Award Program.
}
\fntext[AO]{A. Ovchinnikov was supported by the grants: NSF  CCF-0952591 and  PSC-CUNY  \#60001-40~41.}

\fntext[conference]{The extended abstract \cite{MOCiE} presented proofs of results
in Section \ref{sec:theorem}; those proofs are omitted here.  Much of Section \ref{sec:intro}
appeared as an extended abstract in the unpublished proceedings
of a weeklong workshop at the Mathematisches Forschungsinstitut Oberwolfach.}

\begin{abstract}
Kronecker's Theorem and Rabin's Theorem are fundamental
results about computable fields $F$ and the decidability of the
set of irreducible polynomials over $F$.  We adapt these theorems
to the setting of differential fields $K$, with constrained pairs
of differential polynomials over $K$ assuming the role of the
irreducible polynomials.  We prove that two of the three basic
aspects of Kronecker's Theorem remain true here,
and that the reducibility in one direction (but not the other)
from Rabin's Theorem also continues to hold.
\end{abstract}


\end{frontmatter}
\fontsize{11pt}{11pt}\selectfont

\section{Introduction}
\label{sec:intro}

Differential algebra is the study of differential equations
from a purely algebraic standpoint.  The differential equations
studied use polynomials in a variable $Y$ and its derivatives
$\del Y, \del(\del Y),\ldots$, with coefficients from a specific
field $K$ which admits differentiation on its own elements
via the operator $\del$.  Such a field $K$ is known as a \emph{differential field}:
it is simply a field with one or more additional unary functions
$\del$ on its elements, satisfying the usual properties of derivatives:
$\del (x+y)=(\del x)+(\del y)$ and $\del (x\cdot y) = (x\cdot\del y)+ (y\cdot\del x)$.

It is therefore natural to think of the field elements as functions,
and standard examples include the field $\Q(X)$ of rational functions
in one variable under differentiation $\frac{d}{dX}$, and the field
$\Q(t,\del t,\del^2 t,\ldots)$ with a \emph{differential transcendental} $t$
satisfying no differential equation over the ground field $\Q$.
Additionally, every field becomes a differential field when we set
$\del x=0$ for all $x$ in the field; we call such a differential field
a \emph{constant field}, since an element whose derivative is $0$
is commonly called a \emph{constant}.  We give a good deal
of further background on computability and on differential fields
in Sections \ref{sec:computability}, \ref{sec:fields}, and \ref{sec:background}.
First, though, without dwelling on formal definitions,
we summarize the situation addressed in this article.

Although the natural examples are fields of functions,
the treatment of differential fields regards the field elements
merely as points.  There are strong connections between differential
algebra and algebraic geometry, with such notions as the ring $\KY$ of differential
polynomials (namely the algebraic polynomial ring $K[Y,\del Y,\del^2 Y,\ldots]$,
with each $\del^iY$ treated as a separate variable), differential ideal,
differential variety, and differential Galois group all being direct adaptations
of the corresponding notions from field theory.  

Characteristically, these
concepts behave similarly in both areas, but the differential versions
are often a bit more complicated.  In terms of model theory,
the theories \ACF and \DCF (of algebraically closed fields
and differentially closed fields, respectively, of characteristic $0$)
are both complete and $\omega$-stable with effective quantifier elimination,
but \ACF has Morley rank $1$, whereas \DCF has Morley rank $\omega$.
The higher Morley rank has elicited intense interest in \DCF
from model theorists.

Just as the algebraic closure $\Fbar$ of a field $F$ (of characteristic $0$)
can be defined as the prime model of the theory $\ACF\!\!\cup \Delta(F)$
(where $\Delta(F)$ is the atomic diagram of $F$), the differential closure
$\Khat$ of a differential field $K$ is normally taken to be the prime model
of $\DCF\!\!\cup\Delta(K)$.  This $\Khat$ is unique up to isomorphism over $K$,
but not always minimal:  it is possible for $\Khat$ to embed into itself over $K$
(i.e., fixing $K$ pointwise) with image a proper subset of itself.  This has to do with
the fact that some $1$-types over $K$ are realized infinitely often in $\Khat$,
so that the image of the embedding can omit some of those realizations.

As a prime model, the differential closure realizes exactly those $1$-types
which are principal over $K$, i.e., generated by a single formula with
parameters from $K$.  It therefore omits the type of a differential transcendental
over $K$, since this type is not principal, and so every element of $\Khat$
satisfies some differential polynomial over $K$.  On the other hand, the type
of a \emph{transcendental constant}, i.e., an element $x$ with $\del x=0$
but not algebraic over $K$, is also non-principal and hence is also omitted, even though
such an element would be ``differentially algebraic'' over $K$.

The goal of this article is to adapt the two fundamental
theorems from computable field theory to computable differential fields.
These two theorems, each used very frequently in work on computable fields,
are the following.

\begin{theorem}
%\label{theorem:Kronecker}
[{Kronecker's Theorem (1882); see \cite{K1882} or \cite{E84},
or Theorem \ref{theorem:Kronecker} below}] \hspace{0.1in}
\begin{enumerate}
\item
The field $\Q$ has a splitting algorithm.  That is,
the set of irreducible polynomials in $\Q[X]$,
commonly known as the \emph{splitting set} of $\Q$,
is decidable.
\item
If a computable field $F$ has a splitting algorithm,
so does the field $F(x)$, for every element
$x$ algebraic over $F$ (within a larger computable field).
\item
If a computable field $F$ has a splitting algorithm,
then so does the field $F(t)$, for every element
$t$ transcendental over $F$.
\end{enumerate}
(The algorithms deciding irreducibility in Parts II and III
are different, and no unifying algorithm exists.)
\end{theorem}
\begin{theorem}
%\label{theorem:Rabin}
[{Rabin's Theorem (1960); see \cite{R60},
or Theorem \ref{theorem:Rabin} below}]\hspace{0.1in}
\begin{enumerate}
\item
Every computable field $F$ has a \emph{Rabin embedding},
i.e., a computable field embedding $g:F\to E$ such that
$E$ is a computable, algebraically closed field which is
algebraic over the image $g(F)$.
\item
For every Rabin embedding $g$ of $F$, the image $g(F)$
is Turing-equivalent to the splitting set $S_F$ of $F$.
\end{enumerate}
\end{theorem}

For differential fields, the analogue of the first part of Rabin's Theorem
was proven in 1974 by Harrington \cite{H74}, who showed that for every computable differential
field $K$, there is a computable embedding $g$ of $K$ into a computable,
differentially closed field $L$ such that $L$ is a differential closure
of the image $g(K)$.  Harrington's proof used a different method
from that of Rabin, and therefore did not address the question of the
Turing degree of the image.  Indeed, the first question to address,
in attempting to adapt either of these theorems for differential fields,
is the choice of an appropriate analogue for the splitting set $S_F$
in the differential context.

Kronecker, like many others before and since,
saw the question of reducibility of a polynomial
in $F[X]$ as a natural question to ask, with applications
in a broad range of areas.  However,
with twentieth century model theory, we can specify more exactly the reasons
why it is important.  Specifically, every irreducible polynomial $p(X)\in F[X]$
generates a principal type over the theory $\ACF\!\!\cup\Delta(F)$,
and every principal type is generated by a unique monic irreducible polynomial.
(More exactly, the formula $p(X)=0$ generates such a type.)

On the other hand, no reducible polynomial generates such a type
(with the exception of powers $p(X)^n$ of irreducible polynomials,
in which case $p(X)$ generates the same type).
So the splitting set $S_F$ gives us a list of generators of
principal types, and every element of $\Fbar$ satisfies exactly
one polynomial on the list.  Moreover, since these generating formulas
are quantifier-free, we can readily decide whether a given element satisfies
a given formula from the list or not.  Thus, a decidable splitting set allows us
to identify elements of $\Fbar$ very precisely, up to their orbit over $F$.

From model theory, we find that the set $\TKbar$ of \emph{constrained pairs}
over a differential field $K$ plays the same role for the differential closure.
A pair $(p,q)$ of differential polynomials from $\KY$ is \emph{constrained}
if
\begin{itemize}
\item $p$ is monic and irreducible and of greater order than $q$
(i.e., for some $r$, $p(Y)$ involves $\del^rY$ nontrivially
while $q(Y)\in K[Y,\del Y,\ldots,\del^{r-1}Y]$), and
\item
for every $x,y\in\Khat$, if $p(x)=p(y)=0$ and $q(x)\neq 0\neq q(y)$,
then for every $h\in\KY$, either $h(x)=0 = h(y)$
or $h(y)\neq 0\neq h(x)$.  This says that, if $x$ and $y$ both
\emph{satisfy} the pair $(p,q)$, then the differential fields
$K\la x\ra$ and $K\la y\ra$ that they generate within $\Khat$
must be isomorphic, via an isomorphism fixing $K$ pointwise
and mapping $x$ to $y$.
\end{itemize}  This is sufficient to ensure
that the formula $p(Y)=0\neq q(Y)$ generates a principal type
over $\DCF\!\!\cup\Delta{K}$, and conversely, every principal type
is generated by such a formula with $(p,q)$ a constrained pair.

For these reasons, model theorists have come to see
constrained pairs as the appropriate differential analogue
of irreducible polynomials.  Computability theorists
have reserved their judgment.  They would go along
with the model theorists if it could be shown that
every other list of existential generators
of the principal types over $\DCF\!\!\cup\Delta{K}$
has complexity $\geq_T T_K$, or even if this could be
shown at least for those computable differential fields
appearing commonly in differential algebra; but this is not known.
Essentially, the question turns on whether Kronecker's
Theorem can be adapted and proven in the setting of
computable differential fields $K$ and constraint sets $T_K$.

With this background, we may state our results,
first addressing Rabin's Theorem and then Kronecker's.
\begin{theorem}[{See Theorem \ref{theorem:constraints} and Proposition \ref{prop:FRg}}]
For every embedding $g$ of a computable differential field
as described by Harrington in \cite{H74}, the image
$g(K)$ is Turing-computable from the set $\TKbar$.  So too is algebraic independence
of finite tuples from $\Khat$,
and also the function mapping each $x\in\Khat$ to its
minimal differential polynomial over $K$. However, there do exist
such embeddings $g$ for which $\TKbar$ has no Turing reduction to $g(K)$.
\end{theorem}
\begin{theorem}[{See Theorem \ref{theorem:fgalg} and Theorem \ref{theorem:fgtrans}}]
Let $K$ be a computable nonconstant differential field, and let $z$
be an element of a larger computable differential field $L\supseteq K$
such that $K$ is computably enumerable within $L$.  (So $K\la z\ra$
is also c.e.\ within $L$, and thus has a computable presentation.)
\begin{itemize}
\item
If $z$ is constrained over $K$, then
$\overline{T_{K\la z\ra}}$ is Turing-computable from $\TKbar$.
\item
If $z$ is differentially transcendental and $\Khat$ is not algebraic over $K$, then
$\overline{T_{K\la z\ra}}$ is Turing-computable from $\TKbar$.
\end{itemize}
\end{theorem}
So the last two parts of Kronecker's Theorem hold in most cases.
The first part remains open:  it is unknown whether the set $\overline{T_{\Q}}$
of constrained pairs over the constant differential field $\Q$ is decidable.
We regard this as the most important question currently open in this
area of study.  A positive answer would likely give us a much better
intuition about the structure of various simple differentially
closed fields, well beyond any current understanding.
It would also be desirable to make the failure of the second part of Rabin's
Theorem more precise, by finding sets which are
always equivalent to the Rabin image $g(K)$, and by finding sets
which are always equivalent to $\TKbar$.
\subsection*{Acknowledgments} We are grateful to P.~Cassidy, D.~Marker,
M.~Singer, W.~Sit, and the referee for their helpful suggestions.

\section{Background in Computability}
\label{sec:computability}

We recall here the concepts from computability theory which will be
essential to our work on differential fields.  Computable
functions are defined in \cite{S87}, and indeed,
several very different definitions give rise to the
same class of functions.  Functions on the set $\omega$
of nonnegative integers are usually identified with
their graphs in $\omega^2$, and we then code $\omega^2$
into $\omega$, so that the graph corresponds to a
subset of $\omega$.  Conversely, for our purposes, a subset $A$
of $\omega$ may be identified with its characteristic function $\chi_A$,
and we say that $A$ is \emph{computable} (or \emph{decidable})
if the function $\chi_A$ is computable.

The partial computable functions
(those for which the computation procedure halts
on certain inputs from $\omega$,
but not necessarily on all of them) can be enumerated
effectively, and are usually denoted as $\phi_0,\phi_1,\ldots$,
with the index $e$ coding the program for computing
$\phi_e(x)$ on $x\in\omega$.  The domains of these functions
constitute the \emph{computably enumerable sets}, and we write
$W_e$ for the domain of $\phi_e$.  These are precisely the sets
which are definable by $\Sigma^0_1$ formulas,
i.e., sets of the form
$$ \set{x\in\omega}{\exists y_1\cdots\exists y_m~(x,y_1,\ldots,y_m)\in R},$$
where $m\in\omega$ is arbitrary and $R$ may be
any computable subset of $\omega^{m+1}$.
We usually write ``$\phi_e(x)\converges=y$'' to indicate
that the computation of $\phi_e$ on input $x$ halts
and outputs $y$, and so $\phi_e(x)\converges$ iff $x\in W_e$;
otherwise we write $\phi_e(x)\diverges$.
Also, if the computation halts within $s$ steps,
we write $\phi_{e,s}(x)\converges$.  The set
$W_{e,s}$ is the domain of $\phi_{e,s}$,
so $W_e=\cup_s W_{e,s}$.  Every set $W_{e,s}$
is computable (although the union $W_e$ may not be),
and we take it as a convention of our computations
that only numbers $\leq s$ lie in $W_{e,s}$.

More generally, we define the $\Sigma^0_n$ formulas
by induction on $n$.  The $\Sigma^0_0$ formulas are those
formulas with free variables $x_1,\ldots,x_m$ which define computable
subsets of $\omega^m$ (for any $m\in\omega$).  A $\Pi^0_n$ formula
is the negation of a $\Sigma^0_n$ formula.  Therefore,
the $\Pi^0_0$ formulas are exactly the $\Sigma^0_0$ formulas,
but for $n>0$ this is no longer true:  for instance
a $\Pi^0_1$ formula is universal, in the
same sense that a $\Sigma^0_1$ formula is existential.
A $\Sigma^0_{n+1}$ formula in the variable $x$
is a formula of the form
$$ \exists y_1\cdots\exists y_m~R(x,y_1,\ldots,y_m),$$
where $R$ is a $\Pi^0_n$ formula.  Thus the subscript
counts the number of quantifier alternations.
(We sometimes omit the superscript $0$, which refers
to the fact that we quantify only over natural numbers,
not over sets of naturals, or sets of sets of naturals, etc.)

Turing reducibility and $1$-reducibility are ways of comparing
the complexity of subsets $A,B\subseteq\omega$.  We refer
the reader to \cite{S87} for the definition of Turing reducibility,
and write $A\leq_T B$ to denote that under this reducibility,
$A$ is no more complex than $B$.  $1$-Reducibility is more simply defined.
\begin{definition}
\label{definition:1reducible}
A set $A$ is \emph{$1$-reducible} to a set $B$, written
$A\leq_1 B$, if there exists a computable injective function $f$,
whose domain is all of $\omega$, such that
$$ (\forall x) [x\in A\iff f(x)\in B].$$
\end{definition}
It is well known that, for every $n\in\omega$, there exists a set
$S$ which is \emph{$\Sigma^0_{n+1}$-complete}:  $S$ itself is
$\Sigma^0_{n+1}$, and every $\Sigma^0_{n+1}$ set $T$ has $T\leq_1
S$. Indeed, the Halting Problem, written here as $\emptyset'$, is
$\Sigma^0_1$-complete.  The set $\emptyset''$ is the halting problem
relative to $\emptyset'$, and is $\Sigma^0_2$-complete, and one
iterates this \emph{jump} operation, always taking the halting
problem relative to the previous set, to get the
$\Sigma^0_{n+1}$-complete set $\emptyset'''^{\cdots}$, or
$\emptyset^{(n+1)}$. Likewise, the complement of a
$\Sigma^0_{n+1}$-complete set $S$ is $\Pi^0_{n+1}$-complete. This is
regarded as an exact assessment of the complexity of $S$; among
other things, $\Sigma^0_{n+1}$-completeness ensures
that $S$ is \emph{not} $\Pi^0_{n+1}$, nor
$\Sigma^0_{n}$.  

Note that the class of $\Sigma^0_0$
sets and the class of $\Pi^0_0$ sets coincide:  these are the
computable sets.  For $n>0$, a set which is both $\Sigma^0_n$ and
$\Pi^0_n$ is said to be $\Delta^0_n$. 
The $\Delta^0_{n+1}$ sets are exactly those which are
Turing-reducible to a $\Sigma^0_n$-complete oracle set.  As a
canonical $\Sigma^0_n$-complete set, we usually use
$\emptyset^{(n)}$, the $n$-th jump of the empty set, as defined in
\cite{S87}.

Turing reducibility $\leq_T$ is a partial pre-order on the power set
$\mathcal{P}(\omega)$.  We define $A\equiv_T B$, saying that $A$ and
$B$ are \emph{Turing-equivalent}, if $A\leq_T B$ and $B\leq_T A$.
The equivalence classes under this relation form the \emph{Turing
degrees}, and are partially (but not linearly) ordered by $\leq_T$. In fact, they form
an upper semi-lattice under $\leq_T$, with least element $\bfz$, the
degree of the computable sets, but no greatest element. One often
speaks of a set $A$ as being \emph{computable in a Turing degree
$\bfd$}, meaning that for some (equivalently, for every) $B\in\bfd$
we have $A\leq_T B$.

\section{Background on Fields}
\label{sec:fields}

The next definition arises from the standard notion of a computable structure.
To avoid confusion, we use the domain $\{ x_0,x_1,\ldots\}$
in place of $\omega$.
\begin{definition}
\label{definition:computablefd}
A \emph{computable field} $F$ consists of a set $\set{x_i}{i\in I}$,
where $I$ is an initial segment of $\omega$, such that these elements
form a field with the operations given by Turing-computable functions
$f$ and $g$:
$$  x_i + x_j = x_{f(i,j)}~~~~~~~x_i\cdot x_j = x_{g(i,j)}.$$
\end{definition}

Fr\"ohlich and Shepherdson were the first to consider computable
algebraically closed fields, in \cite{FS56}.  However, the definitive
result on the effectiveness of algebraic closure is Rabin's Theorem.
To state it, we need the natural notions of the root set and the splitting set.
\begin{definition}
\label{definition:RFSF}
Let $F$ be any computable field.  The \emph{root set} $R_F$ of $F$
is the set of all polynomials in $F[X]$ having roots in $F$, and the
\emph{splitting set} $S_F$ is the set of all polynomials in $F[X]$
which are reducible there.  That is,
\begin{align*}
R_F &= \set{p(X)\in F[X]}{(\exists a\in F)~p(a)=0}\\
S_F &= \set{p(X)\in F[X]}{(\exists\text{~nonconstant~}
p_0,p_1\in F[X])~p=p_0\cdot p_1}.
\end{align*}
If $S_F$ is computable, $F$ is said to have a \emph{splitting algorithm}.
\end{definition}
With $F$ computable, $R_F$ and $S_F$ must both be computably enumerable,
being defined by existential conditions.  Theorem \ref{theorem:Rabin}
will show them to be Turing-equivalent.
For most computable fields one meets, both are computable;
the first steps in this direction were taken by
Kronecker in 1882.
\begin{theorem}[{Kronecker's Theorem; see \cite{K1882}}]\hspace{0.1in}
\label{theorem:Kronecker}
\begin{enumerate}
\item[(i)]
$\Q$ has a splitting algorithm.
\item[(ii)]
Let $L$ be a c.e.\ subfield of a computable field $F$.
If $L$ has a splitting algorithm,
then for every $x\in F$ algebraic over $L$,
$L(x)$ also has a splitting algorithm.
The specific decision procedure for $S_{L(x)}$
can be determined from that for $S_L$ and from
the minimal polynomial of $x$ over $L$.
\item[(iii)]
Let $L$ be a c.e.\ subfield of a computable field $F$.
If $L$ has a splitting algorithm,
then for any $x\in F$ transcendental over $L$,
$L(x)$ also has a splitting algorithm, which can be determined
just from that for $L$, given that $x$ is transcendental.
\end{enumerate}
More generally, for any c.e.\ subfield $L$ of a computable
field $F$ and any $x\in F$,
the splitting set of $L(x)$ is Turing-equivalent to
the splitting set for $L$, via reductions uniform in
$x$ and in the minimal polynomial of $x$ over $L$
(or in the knowledge that $x$ is transcendental).
\end{theorem}
The algorithms for algebraic and transcendental
extensions are different, so it is essential
to know whether $x$ is algebraic.  If it is, then from
$S_L$ one can determine the minimal polynomial of $x$.
This yields the following.
\begin{lemma}
\label{lemma:splitalg}
For every computable field $F$ algebraic over
its prime subfield $P$, there is a computable function
which accepts as input any finite tuple
$\xvec=\la x_1,\ldots,x_n\ra$ of elements of $F$ and outputs
an algorithm for computing the splitting set for
the subfield $P[\xvec]$ of $F$.  (We therefore say
that the splitting set of $P[\xvec]$ is
\emph{computable uniformly in $\xvec$}.)
\end{lemma}
\begin{proof}
Clearly there are splitting algorithms for all
finite fields, just by checking all possible factorizations.
(So in fact there is a single algorithm which works
in all positive characteristics.)  In characteristic $0$,
one can readily compute the unique isomorphism
onto the prime subfield $P$ of $F$ from the
computable presentation of $\Q$ for which
Kronecker's splitting algorithm works, and this computable
isomorphism allows us to compute the splitting set of $P$.

The lemma then follows by induction on the size of the tuple
$\xvec=\la x_1,\ldots,x_n\ra$,
using part (ii) of Theorem \ref{theorem:Kronecker}.
Since our $F$ is algebraic over $P$,
we may simply search for a polynomial $p(X)$
with root $x_n$ and coefficients in $P[x_0,\ldots,x_{n-1}]$,
and then factor it, using the splitting algorithm for
$P[x_0,\ldots,x_{n-1}]$ (by inductive hypothesis),
until we have found the minimal polynomial of $x_n$
over $P[x_0,\ldots,x_{n-1}]$.
\end{proof}

At the other extreme, if $F$ is algebraically closed, then clearly both $R_F$
and $S_F$ are computable.  However, there are many computable fields $F$
for which neither $R_F$ nor $S_F$ is computable; see the expository article
\cite[Lemma 7]{Notices08} for a simple example.  Fr\"ohlich and Shepherdson
showed that $R_F$ is computable if and only if $S_F$ is, and Rabin's Theorem
then related them both to a third natural c.e.\ set related to $F$, namely its image
inside its algebraic closure.  (Rabin's work actually ignored Turing degrees,
and focused on $S_F$ rather than $R_F$, but the theorem stated here
follows readily from his proof there.)  More recent works \cite{M09a,S10}
have compared these three sets under stronger reducibilities,
but here, following Rabin, we consider only Turing reducibility,
denoted by $\leq_T$, and Turing equivalence $\equiv_T$.

\begin{theorem}[{Rabin's Theorem; see \cite{R60}}]
\label{theorem:Rabin}
For every computable field $F$, there exist an algebraically closed
computable field $E$ and a computable field homomorphism $g:F\to E$
such that $E$ is algebraic over the image $g(F)$.  Moreover, for every
embedding $g$ satisfying these conditions, the image $g(F)$
is Turing-equivalent to both the root set $R_F$ and the splitting set $S_F$
of the field $F$.
\end{theorem}

We will refer to any embedding $g:F\to E$ satisfying the conditions
from Rabin's Theorem
as a \emph{Rabin embedding} of $F$.  Since this implicitly includes
the presentation of $E$ (which is required by the
conditions to be algebraically closed),
a Rabin embedding is essentially a presentation of the algebraic
closure of $F$, with $F$ as a specific, but perhaps undecidable, subfield.

When we begin to consider polynomials in several variables,
the connection between $R_F$ and $S_F$ breaks down.
\begin{theorem}[{see \cite{FJ86}}]
\label{theorem:FandJ}
Suppose that $F$ is a computable field.  Then for every $n$,
the set of irreducible polynomials in the ring $F[ X_0,\ldots,X_n]$
is computable in an oracle for the splitting set $S_F$,
via a Turing reduction uniform in $n$.

Hence, in a computable differential field $K$,
it is decidable in $S_K$ whether a differential polynomial
$p(Y)\in K\{ Y\}$, viewed as an algebraic polynomial
over $K$ in $Y,\del Y,\del^2Y,\ldots$, is irreducible.
(Differential polynomials are described in the next section.)
\end{theorem}
For a proof, see \cite[\S 19]{FJ86}.  In contrast,
the decidability of the existence of rational solutions to arbitrary
polynomials in $\Q[X_1,X_2,\ldots ]$ remains an open question:
this is \emph{Hilbert's Tenth Problem for $\Q$}, the subject
of much study.  At present, it is not clear that these questions
impinge on single-variable problems for differential fields,
but since a differential polynomial can be viewed as an algebraic
polynomial in several variables, it is not implausible that
a connection might exist.

We now turn to questions of algebraic dependence in fields.
Predictably, these issues are closely tied to transcendence bases.
\begin{definition}
\label{definition:dependence}
For a computable field $F$ with computably enumerable subfield $E$,
the \emph{algebraic dependence set} $A_{F/E}$ is the set of all
finite tuples of $F$ algebraically dependent over $E$:
$$ A_{F/E} =\set{(x_1,\ldots,x_n)\in F^{<\omega}}{%n\in\omega~\&~
(\exists\text{~nonzero~}p\in E[X_1,\ldots,X_n])~p(\xvec)=0}.$$
Likewise, for any computable $F$-vector space $V$ (including computable
field extensions of $F$), the \emph{linear dependence set} is
$$ L_V =\set{\{ v_1,\ldots,v_n\}\subseteq V}{%n\in\omega~\&~
(\exists \la c_1,\ldots,c_n\ra\in F^n-\{\la 0,\ldots,0\ra\})~\sum\nolimits_{i\leq n} a_iv_i=0}.$$
\end{definition}

Below, when considering a differential field $K$ within its differential
closure $\Khat$, we will often want to consider this set for $\Khat$ over $K$,
and we will write
$$ D_K = A_{\Khat/K} $$
for the set of all finite subsets of $\Khat$ algebraically dependent over $K$.
The following lemma is considered in more depth in \cite{MM13}.

\begin{lemma}
\label{lemma:algdep}
For every computable field $F$ and
every computably enumerable subfield $E$, there is a transcendence
basis $B_{F/E}$ for $F$ over $E$ computable using the set $A_{F/E}$ as an oracle.
Conversely, for every transcendence basis $B$ for $F$ over $E$,
we have $A_{F/E}\leq_T B$.
\end{lemma}
\begin{proof}
We define the \emph{canonical transcendence basis} $B_{F/E}$ for $F$
over $E$ as $\cup_s B_s$, where $B_0=\emptyset$ and
$$ B_{s+1}=\left\{\begin{array}{cl} B_s\cup\{ x_s\}, & \text{if
this set is algebraically independent over $E$,}\\
B_s, & \text{if not.}\end{array}\right.$$
Clearly $B_{F/E}\leq_T A_{F/E}$ (and so, by the next paragraph,
$B_{F/E}\equiv_T A_{F/E}$).

If $F$ has finite transcendence degree over $E$,
then every transcendence basis is computable,
with no oracle at all.  So we assume the transcendence degree to be infinite.
If $B$ is a transcendence basis and $$S=\{ x_1,\ldots,x_n\}\subseteq F,$$
then $S$ is algebraically independent iff there exist an $r\geq n$,
an $r$-element subset $B_0=\{ b_1,\ldots,b_r\}\subseteq B$,
and elements $y_{n+1},\ldots,y_r\in F$ such that every element
in each of the sets $B_0$ and $$\{ x_1,\ldots,x_n,y_{n+1},\ldots,y_r\}$$
is algebraic over the other set.  This condition is $\Sigma^0_1$ relative to $B$.
Of course, algebraic dependence of $S$ is $\Sigma^0_1$ (without any oracle),
so membership of $S$ in $A_{F/E}$ is decidable from $B$.
\end{proof}
An exactly analogous proof likewise holds for bases of vector spaces.
\begin{lemma}
\label{lemma:lindep} For every computable field $F$ and every
computable $F$-vector space $V$ (such as a computable field
extension of $F$), the relation $L_V$ of linear dependence
is computable relative to any basis for $V$
over $F$.  Conversely, this relation itself computes a basis.
\end{lemma}

\section{Computable Differential Fields}
\label{sec:background}

Differential fields are a generalization of fields, in which the field
elements are often viewed as functions.
The elements are not treated as functions, but the differential operator(s)
on them are modeled on the usual notion of differentiation of functions.
\begin{definition}
\label{definition:difffield}
A \emph{differential field} is a field $K$ with one or more additional
unary functions $\del_i$ satisfying the following two axioms for all
$x,y\in K$:
$$ \del_i(x+y)=\del_i x + \del_i y~~~~~~~
\del_i(x\cdot y)=(x\cdot\del_i y) + (y\cdot\del_i x).$$
The \emph{constants} of $K$ are those $x$ such that, for all $i$,
$\del_i x=0$.  They form a differential subfield $C_K$ of $K$.
\end{definition}
So every field can be made into a differential field by adjoining
the zero operator $\del x=0$.  For a more common example,
consider the field $F(X_1,\ldots,X_n)$ of rational functions over
a field $F$, with the partial derivative operators
$\del_i=\frac{\partial}{\partial X_i}$.
We will be concerned only with \emph{ordinary differential fields},
i.e., those with a single differential operator $\del$.

A differential field $K$ gives rise to a ring $\KY$ of \emph{differential polynomials}.
This can be described as $$K\big[Y,\del Y,\del^2Y,\ldots\big],$$
the ring of algebraic polynomials in the infinitely many variables shown.
However, for any differential polynomial $p$ and any single $y\in K$,
it makes sense to speak of $p(y)$, by which we mean the
element $$p\big(y,\del y,\del^2 y,\ldots\big)\in K$$ calculated using $\del$
and the field operations of $K$.  Likewise, in any differential field
extension $L$ of $K$, $p(x)$ will be an element of $L$ for every $x\in L$.

We can similarly discuss the derivative of a polynomial, bearing in
mind that the coefficients are elements of $K$, not necessarily constants,
and may require the Leibniz Rule.  For instance, if
$$p(Y) = a{\left(\del Y\right)}^2 + bY{\left(\del^2 Y\right)}$$ with $a,b\in K$, then
$$ (\del p)(Y) = (\del a)(\del Y)^2 +2a(\del Y){\left(\del^2 Y\right)}
+(\del b)Y{\left(\del^2Y\right)}+b(\del Y){\left(\del^2 Y\right)} +bY{\left(\del^3 Y\right)}.$$
The \emph{order}
of $p\in K\{ Y\}$ is the greatest $r\geq 0$ such that $\del^rY$ appears
nontrivially in $p(Y)$.  So, in the example above, $p(Y)$ has order $2$
and $\del p(Y)$ has order $3$.  The algebraic polynomials (in $K[Y]$) 
of positive degree in $Y$ have order $0$.  By convention the zero polynomial has
order $-\infty$, and all other algebraic polynomials of degree $0$ have order $-1$.

Just as polynomials in $F[X]$ are ranked by their degree,
differential polynomials in $\KY$ are ranked as well.
First, if $p,q\in\KY$ and $\ord(p)<\ord(q)$, then $p$
has lower rank than $q$.  Second, if $\ord(p)=\ord(q)=r$
but $\del^rY$ has lesser degree in $p$ than in $q$,
then $p$ has lower rank than $q$.  This is not the entire notion
of rank, but it is as much as we need
in this paper:
rank is the lexicographic
order on the pair $\big(\ord(p),\text{deg}_{\del^{\ord(p)}Y}(p)\big)$,
hence of order type $\omega^2$.

The notion of a computable differential field extends that of a computable field
in the natural way.
\begin{definition}
\label{definition:computabledifffd}
A \emph{computable differential field} is a computable field with one
or more differential operators $\del$ as in Definition \ref{definition:difffield},
each of which is likewise given by some Turing-computable function $h$
with $\del(x_i)=x_{h(i)}$ (where $\{ x_0,x_1,\ldots\}$ is again
the set of elements of the differential field, as in Definition \ref{definition:computablefd}).
\end{definition}

As we shift to consideration of differential fields, we must first
consider the analogy between algebraic closures of fields and differential
closures of differential fields.  The theory $\DCF$
of differentially closed fields $K$ of characteristic $0$
is a complete theory, and was axiomatized by Blum
(see e.g.\ \cite{B77}) using the axioms
for differential fields of characteristic $0$,
along with axioms stating that, for every pair of nonzero differential
polynomials $p,q\in K\{ Y\}$ with $\ord(p) >\ord(q)$, there exists
some $y\in K$ with $p(y)=0\neq q(y)$.  
(By convention, all nonzero constant polynomials have order $-1$.
Blum's axioms therefore include formulas saying that $K$ is
algebraically closed, by taking $q=1$ and
arbitrary nonconstant $p\in K[Y]$.)

For a differential field $K$ with extensions containing
elements $x_0$ and $x_1$, we will write $$x_0 \cong_K x_1$$
to denote that $K\la x_0\ra\cong K\la x_1\ra$ via an isomorphism
fixing $K$ pointwise and sending $x_0$ to $x_1$.
This is equivalent to the property that, for all
$h\in K\{ Y\}$, $$h(x_0)=0 \iff h(x_1)=0;$$ a model theorist
would say that $x_0$ and $x_1$ realize the same atomic type over $K$.
The same notation $x_0\cong_F x_1$ could apply to elements
of field extensions of a field $F$, for which the equivalent
property would involve only algebraic polynomials $h\in K[Y]$.

Let $K\subseteq L$ be an extension of differential fields.  An element
$z\in L$ is \emph{constrained over $K$} if $z$ satisfies some
\emph{constrained pair} over $K$, as defined here.  The terminology
of ``principal types'' reflects the model theory behind the definition.
In fact, the type generated is principal over the theory \DCF.

\begin{definition}
\label{definition:constraint}
Let $K$ be a differential field.
A pair $(p,q)$ of differential polynomials in $K\{ Y\}$
\emph{generates a principal type} if, for all differential field extensions $L_0$
and $L_1$ of $K$ and all $x_i\in L_i$ such that $p(x_i)=0\neq q(x_i)$,
we have $x_0\cong_K x_1$.
The pair $(p,q)$ is \emph{constrained} for $K$ if $p$ is
monic and algebraically irreducible over $K$, with $\ord(q)<\ord(p)$, and
$(p,q)$ generates a principal type.
Elements $x$ in an extension of $K$ with $p(x)=0\neq q(x)$
are said to \emph{satisfy the constrained pair $(p,q)$}. We let
\begin{align*}
T_K &=\set{(p,q)\in(K\{ Y\})^2}{(p,q)\text{~is not a constraint}}\\
&= \set{(p,q)}{{\left(\exists x,y\in\Khat\right)}(\exists h\in\KY)[
h(x)=0\neq h(y)~\&~x,y\text{~satisfy~}(p,q)]}.
\end{align*}
(The second of these equivalent definitions should logically follow Theorem
\ref{theorem:cclosure}, where we define the differential closure $\Khat$  of $K$,
and Proposition \ref{prop:noniso}, which establishes the equivalence.)
Thus $\TKbar$ is the \emph{constraint set} for $K$.  If $T_K$ is computable,
we say that $K$ has a \emph{constraint algorithm}.
\end{definition}

The broad intention
is to quantify over all $x$ and $y$ in all differential field extensions of $K$.
However, since the definition considers only those $x_0$ and $x_1$
satisfying the constrained pair, it turns out to be equivalent to quantify
over all $x$ and $y$ in a differential closure $\Khat$ of $K$.
Once we state Harrington's Theorem below, the quantification
will just be over elements of $\omega$, and so the second definition
of $T_K$ above uses a $\Sigma^0_1$ formula, provided that $K$ is computable.
(The first definition can also be seen directly to be $\Sigma^0_1$:  the
model-theoretic notion of generating a principal type is $\Pi^0_1$, since
the theory $\DCF\cup\Delta(K)$ described below is complete and decidable.)
This was our reason for defining $T_K$ to be the \emph{complement} of the
constraint set:  we thus parallel the notation $R_F$ and $S_F$
for the corresponding sets for fields, in that all are existential definitions.
(For this purpose we eschew the symbol $C_K$, which is already widely used
to denote the constant subfield of $K$.)

Definition \ref{definition:constraint}
parallels the definition of the splitting set $S_F$
in function if not in form.  For fields $F$, irreducible polynomials $p(X)$
have exactly the same property:  if $p(x_0)=p(x_1)=0$ (for $x_0$ and $x_1$
in any field extensions of $F$), then $x_0\cong_F x_1$ (that is,
$F(x_0)\cong F(x_1)$ via an $F$-isomorphism mapping $x_0$ to $x_1$).
So $T_K$ is indeed the analogue of $S_F$:
both are $\Sigma^0_1$ sets, given that $K$ and $F$ are both computable,
and both are the negations of the properties we need to produce
isomorphic extensions.

If $x\in L$ is constrained over $K$ by $(p,q)$, then there exists a
differential subfield of $L$, extending $K$ and containing $x$,
whose transcendence degree as a field extension of $K$ is finite.
Indeed, writing $K\la x\ra$ for the smallest differential subfield
of $L$ containing $x$ and all of $K$, we see that the transcendence
degree of $K\la x\ra$ over $K$ is the smallest order $r$ of any
nonzero element $p$ of $K\{ Y\}$ for which $x$ is a zero, and that $$\left\{ x,\del
x,\ldots,\del^{r-1}x\right\}$$ forms a transcendence basis for $K\la
x\ra$ as a field extension of $K$. The unique irreducible $p$ of smallest order
is called the \emph{minimal differential polynomial of $x$}, and its order $r$ is
called the \emph{order of $x$}.
(Fact \ref{fact:order} and Definition \ref{definition:order} will elaborate on this.
For more general results, see \cite[II.12, Theorem~6(d)]{KolchinBook1973},
and also \cite[Lemma~6.12]{Poizat} for ordinary differential fields.)
The elements of $L$ which are constrained over $K$ turn
out to form a differential field in their own right.
If this subfield is all of $L$, then $L$ itself is said to be a
\emph{constrained extension} of $K$.

An algebraic closure $\Fbar$ of a field $F$ is an algebraically
closed field which extends $F$ and is algebraic over it.  Of course,
one soon proves that this field is unique up to isomorphism over $F$
(that is, up to isomorphisms which restrict to the identity on the
common subfield $F$). On the other hand, each $F$ has many
algebraically closed extensions; the algebraic closure is just the
smallest of them.  Likewise, each differential field $K$ has many
differentially closed field extensions; a \emph{differential
closure} of $K$ is such an extension which is constrained over $K$.

Model-theoretically, the two situations are closely analogous:
the algebraic closure $\Fbar$ of $F$ (of characteristic $0$) is the prime
model of the theory $\textbf{ACF}_0\cup \Delta (F)$,
given by extending the language to include constants for all elements
of $F$ and adjoining to $\textbf{ACF}_0$ the atomic diagram
$\Delta (F)$.  Likewise, the differential closure $\Khat$ of $K$
is the prime model of $\DCF\cup\Delta (K)$.

As with fields, the differential closure $\Khat$ of $K$ turns out to be
unique up to isomorphism over $K$.  On the other hand,
$\Khat$ is generally not \emph{minimal}:
there exist differential field embeddings of $\Khat$ into itself
over $K$ whose images are proper differential subfields of $\Khat$.
This provides a first contrast between \DCF\ and \textbf{ACF}$_0$,
since the corresponding statement about algebraic closures is false.

\begin{theorem}[{\cite{KolchinConstrained,Robinson,Shelah}}]
\label{theorem:cclosure}
For every differential field $K$ of characteristic $0$,
there exists a differential field extension $\Khat\supseteq K$
which satisfies the axiom set \DCF\ and has the property
that every element of $\Khat$ satisfies some constrained pair from $\TKbar$.
We refer to $\Khat$ as the \emph{differential closure} of $K$,
since it can be shown to be unique up to isomorphisms which fix $K$ pointwise.
\end{theorem}

Our next fact was already cited in Definition \ref{definition:constraint} and is standard.

\begin{proposition}
\label{prop:noniso}
Let $x$ and $y$ lie in any two differential field extensions of $K$.
Then $x\cong_K y$ iff, for all $h\in\KY$, we have
$h(x)=0$ iff $h(y)=0$.
\end{proposition}

With this much information in hand, we can now state the parallel
to the first half of Theorem \ref{theorem:Rabin}.

\begin{theorem}[{\cite[Corollary~3]{H74}}]
\label{theorem:Harrington}
For every computable differential field $K$, there exists a differentially
closed computable differential field $L$ and a computable differential
field homomorphism $g:K\to L$ such that $L$ is constrained over
the image $g(K)$.
\end{theorem}
For the sake of uniform terminology, we continue to refer to
a computable function $g$ as in Theorem \ref{theorem:Harrington}
as a \emph{Rabin embedding} for the differential field $K$.

Harrington actually proves the existence of a computable structure
$L$ which is the prime model of the theory $T$ generated by \DCF \
and the atomic diagram of $K$. Thus $L$ is a computable structure in
the language $\L'$ in which the language of differential fields is
augmented by constants for each element of $K$.  The embedding of
$K$ into $L$ is accomplished by finding, for any given $x\in K$, the
unique element $y\in L$ which is equal to the constant symbol for
$x$.  Clearly this is a computable process, since $L$ is a
computable $\L'$-structure, and so we have our embedding of $K$ into
$L$.  Since $L$ is the prime model of $T$, it must be constrained
over $K$: otherwise it could not embed into the constrained closure,
which is another model of $T$.  So $L$ satisfies the definition of
the differential closure of $K$, modulo the computable embedding.

The root set and splitting set of a differential field $K$ are still defined,
of course, just as for any other field.  However,
with the differential operator $\del$ now in the language,
several other sets can be defined along the same lines
and are of potential use as we attempt to adapt Rabin's Theorem.
The most important of these is the constraint set $\TKbar$,
from Definition \ref{definition:constraint},
which is analogous in several ways to the splitting set and
will be the focus of our attempts to adapt Kronecker's Theorem
(Theorem \ref{theorem:Kronecker} above)
to differential fields.

We will also need a version of the Theorem of the Primitive Element
for differential fields.  This was provided long since by Kolchin.
\begin{theorem}[{\cite[p.~728]{K42}}]
\label{theorem:Kolchin}
Assume that an ordinary differential field $F$ contains an element
$x$ with $\del x\neq 0$.  If $E$ is a differential subfield of the
differential closure $\Fhat$ with $F \subset E$, and $E$ is generated
(as a differential field over $F$) by finitely many elements, then there is a single
element of $E$ that generates all of $E$ as a differential field over $F$.
\end{theorem}
Kolchin extended this theorem to partial
differential fields with $m$ derivations: the generalized condition there
is the existence of $m$ elements whose Jacobian is nonzero.
He offered counterexamples in the case where $\del$ is the
zero derivation on $F$, but in the counterexamples, the generators
of $E$ are constants which are algebraically independent over $F$
and therefore do not lie in $\Fhat$,
although they are differentially algebraic over $F$.  (At that time, the definition
of differential closure was not yet well established.)  It remains open whether
the theorem as stated here holds for constant differential fields as well.

One sees readily that Theorem \ref{theorem:Kolchin} can be carried out effectively:
given an enumeration of $F$ within $E$, along with the finitely many
generators of $E$ over $F$, it is easy to find a single generator as described
in the theorem, simply by enumerating the elements generated over $F$
by each single $x\in E$ until some such $x$ is seen to generate
all of the finitely many given generators of $E$.  Of course, without Theorem
\ref{theorem:Kolchin}, one would not be sure whether this process would ever halt.

\section{Constrained Pairs}
\label{sec:constraints}

For fields, the usefulness of the set $S_F$ of reducible polynomials
in $S_F$ is that the irreducible polynomials in $F[X]$ exactly
define the isolated $1$-types over $\textbf{ACF}_0\cup\Delta(F)$.
That is, if
$x$ and $y$ lie in field extensions of $F$ and satisfy the same
irreducible $p(X)\in F[X]$, then they generate isomorphic subfields:
$F(x)\cong F(y)$, via an isomorphism fixing $F$ pointwise and
mapping $x$ to $y$.  (Reducible polynomials in $F[X]$ fail to
have the same property, except for the special case of a power of an
irreducible polynomial, which can be recognized effectively using
the formal derivative and the Euclidean algorithm.)  Moreover, satisfaction of a particular
irreducible polynomial $p(X)$ by a particular $x$
is decidable, in a computable field extending $F$:
the formula $p(X)=0$ generating the $1$-type is quantifier-free,
and the field operations are computable.
Of course, since  $\textbf{ACF}_0\cup\Delta(F)$ has
effective quantifier elimination, we could start with any generating formula
and find a quantifier-free generating formula.

 The important
point is that we have a list of formulas,
each of which generates a principal type, and such that every principal
type is generated by a formula on the list.  Theorem
\ref{theorem:Kronecker} says that this list can be given effectively
for the field $\Q$, and that the effectiveness carries over
to finitely generated computable field extensions.
This is the theorem which we wish to adapt for differential fields $K$,
with the constraint set $\TKbar$ in place of the set $\overline{S_F}$
of irreducible polynomials in $F[X]$.  After some further preliminaries,
we will prove that the natural adaptations of parts (ii) and (iii)
of Theorem \ref{theorem:Kronecker} do indeed carry over to differential fields.
Whether part (i) can likewise be adapted remains an open question.

\begin{fact}
\label{fact:order}
Let $K$ be a differential field.  Then for each $x\in\Khat$,
there is exactly one $p\in K\{ Y\}$ such that $x$
satisfies a constrained pair of the form $(p,q)\in\TKbar$.
(Recall that by definition $p$ is required to be monic and irreducible.)
Moreover, $\ord(p)$ is least among the orders of all nonzero
differential polynomials in the radical differential ideal
$I_K(x)$ of $x$ within $K\{ Y\}$:
$$ I_K(x) = \set{p\in K\{ Y\}}{p(x)=0},$$
and $\deg{p}$ is the least degree of $\del^{\ord(p)}Y$
in any polynomial in $K\{ Y\}$ of order $\ord(p)$ with root $x$.
\end{fact}
\begin{proof}
Since $\Khat$ is constrained over $K$, each $x\in\Khat$
satisfies at least one constrained pair $(p,q)\in\TKbar$.
Set $r=\ord(p)$, and 
suppose there were a nonzero $\ptilde(Y)\in I_K(x)$
with $\ord(\ptilde) < r$.
By Blum's axioms for $\textbf{DCF}_0$, there would exist
$y\in\Khat$ with $$p(y)=0\neq q(y)\cdot\ptilde(y),$$
since the product $(q\cdot\ptilde)$ has order $< r$.
But then $y$ also satisfies $(p,q)$, yet $\ptilde(y)\neq 0=\ptilde(x)$,
so that $K\la x\ra\not\cong K\la y\ra$.  This would contradict
Definition \ref{definition:constraint}.  Hence $r$ is the least order
of any nonzero differential polynomial with root $x$.

It follows from minimality of $r$ that
$\big\{ x,\del x,\ldots,\del^{r-1}x\big\}$ is algebraically independent over $K$.
Next we claim
that the minimal polynomial of $\del^rx$ over $K\big(x,\ldots,\del^{r-1}x\big)$ must equal
$$p\big(x,\del x,\ldots,\del^{r-1}x,Y\big).$$  Clearly the minimal polynomial must divide
$p\big(x,\del x,\ldots,\del^{r-1}x,Y\big)$, and if it were a proper factor,
then by Blum's axioms for $\DCF$, the quotient of
$p\big(x,\del x,\ldots,\del^{r-1}x,Y\big)$ by the minimal polynomial
would have a root $y$ with $q(y)\neq 0$ which would not be a root of
the minimal polynomial, so that $K\la x\ra\not\cong K\la y\ra$,
contradicting Definition \ref{definition:constraint}.  (We used here the fact
that with $p(Y,\del Y,\ldots,\del^{r}Y)$ irreducible
in $K\big[Y,\del Y,\ldots,\del^r Y\big]$, $$p\big(x,\del x,\ldots,\del^{r-1}x,Y\big)$$
cannot be a power of the minimal polynomial.)
Thus $p$ is the \emph{minimal differential polynomial} of $x$ over $K$.
\end{proof}

\begin{definition}
\label{definition:order} If $K\subseteq L$ is an extension of differential
fields, then for each $x\in L$, we define $\ord_K(x)=\ord(p)$,
where $(p,q)$ is any constrained pair in $\TKbar$ satisfied by
$x$.  
Notice that in the differential closure of $K$, every element has
a well-defined finite order over $K$, by Fact \ref{fact:order}.
(A more general definition of order for elements of differential field
extensions sets $\ord_K(x)\leq\omega$ be the transcendence
degree of $K\la x\ra$ as a field extension of $K$.)
\end{definition}

In fact, the irreducibility of $p(Y)$ is barely necessary in Definition
\ref{definition:constraint}.  The condition that $K\la x\ra\cong K\la y\ra$
for all $x$ and $y$ satisfying the constrained pair shows that $p(Y)$
cannot factor as the product of two distinct differential polynomials.
The only reason for requiring irreducibility of $p$ is
to rule out the possibility of $p$ being a perfect square,
cube, etc.\ in $K\{ Y\}$.  If these were allowed,
the uniqueness in Fact \ref{fact:order} would no longer hold.

\begin{proposition}
\label{prop:theq}
Let $p,q,\qtilde\in\KY$, with $(p,q)$ 
in the constraint set $\TKbar$ and $\ord(\qtilde)<\ord(p)$.
Then $(p,\qtilde)\in\TKbar$ iff,
for every $x\in\Khat$, $x$ satisfies $(p,q)$ iff $x$ satisfies $(p,\qtilde)$.
\end{proposition}
\begin{proof}
First suppose $(p,\qtilde)\in\TKbar$. If $x$ satisfies $(p,q)$,
then $p$ is the minimal differential polynomial of $x$ over $K$
(by Fact \ref{fact:order}), and so $\qtilde(x)\neq 0$.
Likewise, if $x$ satisfies $(p,\qtilde)$, then it satisfies $(p,q)$ as well.
Conversely, if the second condition holds, then every $x$ and $y$
satisfying $(p,\qtilde)$ also satisfy $(p,q)$, hence have $x\cong_K y$,
putting $(p,\qtilde)\in\TKbar$.
\end{proof}

It is quickly seen that if $(p,q)\in\TKbar$,
then also $(p,q\cdot h)\in\TKbar$ for every $h\in K\{ Y\}$.
So Proposition \ref{prop:theq} is nontrivial, and
the constrained pair satisfied by an $x\in\Khat$ is not unique,
although its first component is unique.  The
proposition shows that the first component essentially determines
the constrained pair:  two constrained pairs with the same first component
define the same set.  On the other hand, not all
monic irreducible differential polynomials $p$ can be the
first component of a constrained pair; the polynomial
$p(Y)=\del Y$ is a simple counterexample.  In Section
\ref{sec:constrainability} we will address the question of \emph{constrainability}:
for which $p\in\KY$ does there exist some $q\in\KY$ with $(p,q)\in\TKbar$?

\section{Decidability in the Constraint Set and Rabin Image}
\label{sec:theorem}
\subsection{Decidability in the Constraint Set}
\label{subsec:constraint}

The two theorems in this section were proven in
\cite{MOCiE}, a preliminary report on the work in this article.
They address the adaptation of Rabin's Theorem
to the context of differential fields.  The proofs are
straightforward, and we do not repeat them here.

\begin{theorem}[{\cite[Theorem~10]{MOCiE}}]
\label{theorem:constraints}
Let $K$ be any computable differential field,
and $g:K\to\Khat$
a (differential) Rabin embedding of $K$.  Then all of the following
are computable in an oracle for the constraint set $\TKbar$:
the splitting set $S_K$, the Rabin image $g(K)$,
and the order function $\ord_K$ on $\Khat$.
If additionally our derivation is nontrivial, then the set
$D_K$ of finite subsets of $\Khat$ algebraically dependent
over $g(K)$ is also computable in a $\TKbar$-oracle.
\end{theorem}

In particular, the Rabin image $g(K)$ is computable in a $T_K$-oracle.
This means that $\left(g(K)\cap C_{\Khat}\right)$ is a $T_K$-computable
subfield of the constant field $C_{\Khat}$, which in turn is a
computable subfield of $\Khat$. Indeed, the restriction of $g$ to
$C_K$ is a Rabin embedding of the computable field $C_K$ into its
algebraic closure $C_{\Khat}$, in the sense of Theorem
\ref{theorem:Rabin}, the original theorem of Rabin for fields.

Therefore, if $C$ is any computable field without a splitting algorithm,
we can set $K=C$ to be a differential field with $C_K=C$
(by using the zero derivation).  Theorem \ref{theorem:Harrington}
gives a Rabin embedding $g$ of this differential field $K$
into a computable presentation of $\Khat$.
Theorem \ref{theorem:Rabin} shows that $g(K)=g(C_K)$
is noncomputable within the computable subfield $C_{\Khat}$,
and therefore must be noncomputable within $\Khat$ itself.
Finally, Theorem \ref{theorem:constraints} shows that the constraint
set $\TKbar$ of this differential field $K$ was noncomputable.

So there do exist computable differential fields, even with
the simplest possible derivation, for which the constraint set
is noncomputable.  In the opposite direction, it is certainly
true that if $K$ itself is already differentially closed,
then its constraint set is computable, since the constrained pairs are
exactly those pairs of the form $(Y-a,b)$ with $a,b\in K$ and $b\neq 0$.
(Such a pair is satisfied by exactly one $x\in K$, hence by exactly
one $x\in\Khat=g(K)$, using the identity function as the Rabin
embedding $g$. Thus it trivially satisfies Definition
\ref{definition:constraint}.) 

We do not yet know any examples of
computable differential fields which have computable constraint set,
yet are not differentially closed.  The decidability of the
constraint set is a significant open problem for computable
differential fields in general.  So likewise is the decidability of
constrainability: for which $p\in K\{ Y\}$ does there exist a $q$
with $(p,q)\in\TKbar$?  We address this question in Section
\ref{sec:constrainability}.  The comments in the proof of Theorem
\ref{theorem:constraints} in \cite{MOCiE} make it clear that $p(Y)=\del Y$ is an example
of a differential polynomial which is not constrainable.

\subsection{Decidability in the Rabin Image}
\label{subsec:g(K)}

Rabin's Theorem for fields, stated above as Theorem \ref{theorem:Rabin},
gave the Turing equivalence of the Rabin image $g(F)$ and the
splitting set $S_F$. 
Our principal analogue
of $S_F$ for differential fields $K$ is the set $T_K$,
and Theorem \ref{theorem:constraints} makes some headway in identifying
sets, including $g(K)$ but not only that set, whose join is Turing-equivalent
to $T_K$.  It is also natural to ask about Rabin's Theorem from the other
side:  what set (or what join of sets)
must be Turing-equivalent to $g(K)$?
We now present one step towards an answer to that question,
using the notion of a \emph{linear} differential polynomial in $K\{ Y\}$.
Recall that ``linear'' here is used in exactly the sense of
field theory:  the polynomial has a constant term, and every other
term is of the form $a\del^iY$, for some $a\in K$ and $i\in\omega$.
If the constant term is $0$, then the polynomial is
\emph{homogeneous} as well, every term having degree $1$.

The solutions in $\Khat$
of a homogeneous linear polynomial $p(Y)$ of order $r$
are well known to form an $r$-dimensional vector space
over the constant field $C_{\Khat}$.  By additivity,
the solutions in $\Khat$ to any linear polynomial of order $r$
then form the translation of such a vector space by a single
root $x$ of $p(Y)$.  Of course, not all of the solutions
in $\Khat$ need lie in $K$:  the solutions to $p(Y)=0$
in $K$ (if any exist) form the translation of a vector space
over $C_K$ of dimension~$\leq r$.

\begin{proposition}
\label{prop:FRg}\cite[Proposition~11]{MOCiE}
In a computable differential field $K$ whose field $C_K$
of constants is algebraically closed, the full linear root set $FR_K$:
$$\set{\text{linear~}p(Y)\in K\{ Y\} }{p(Y)=0\text{~has
solution space in $K$ of dim~}\ord(p)},$$
is computable from an oracle for the image $g(K)$ of $K$
in any computable differential
closure $\Khat$ of $K$ under any Rabin embedding $g$.
Moreover, the Turing reduction is uniform in indices for $\Khat$ and $g$.
\end{proposition}

It would be of interest to try to extend this result to the case where
$C_K$ need not be algebraically closed, and/or to the situation
involving the differential closure of an extension of $K$ by
algebraically independent constants.

\section{The Constrainability Set}
\label{sec:constrainability}

We now address the question of whether a given differential
polynomial $p(Y)\in\KY$ is \emph{constrainable}
within its computable differential field $K$.
This question asks whether there exists some $q\in\KY$
such that $(p,q)$ forms a constrained pair, i.e., 
such that $(p,q)\notin T_K$.  Such a $q$ is called a
\emph{constraint} on $p$.
\begin{definition}
\label{definition:constrainability}
For a differential field $K$, the set of
\emph{unconstrainable} differential polynomials is the set
$$ U_K =\set{p\in\KY}{(\forall q\in\KY) (p,q)\in T_K}.$$
Any $p(Y)\notin U_K$ is said to be \emph{constrainable}.
\end{definition}

As the natural c.e.\ and co-c.e.\ sets for fields are all named
in their existential forms ($R_K$, $S_K$, $T_K$, and $D_K$),
we propose that related $\Sigma^0_2$ and $\Pi^0_2$ sets
should always be named in their $\Pi^0_2$ (that is, $\forall\exists$) forms.
As in the case of $U_K$, this will often involve
a single $\forall$ quantifier over an existential set,
which makes a simple and natural definition.
An alternative definition follows from our next proposition.
\begin{fact}
\label{fact:mindiffpoly}
A differential polynomial $p\in\KY$ is constrainable over $K$
iff $p$ is the minimal differential polynomial of some $x$
in the differential closure $\Khat$ of $K$.
\end{fact}
\begin{proof}
Every constrained pair $(p,q)$ is satisfied by some $x\in\Khat$.
Fact \ref{fact:order} shows that $p$
must be the minimal differential polynomial of this $x$,
and also shows that the minimal differential
polynomial of each $x\in\Khat$ is the first component
of the constrained pair satisfied by that $x$.
\end{proof}

This alternative definition of constrainability is no simpler,
however:  it says that there exists an $x$ satisfying $p$
such that no polynomial in $\KY$ of lesser rank than $p$
can have $x$ as a zero.
We now show that $U_K$ cannot be defined
by any formula simpler than $\Pi^0_2$.

\begin{theorem}
\label{theorem:Sigma2}
There exists a computable differential field $K$
in which the set of constrainable polynomials
$p\in\KY$ is $\Sigma^0_2$-complete.
\end{theorem}
\begin{proof}
This proof evolved out of a suggestion by David Marker.
The key to the $\Pi^0_2$-completeness is Theorem 6.2
from \cite[p.\ 73]{MMP06}, which introduces the differential polynomials
$$p_e(Y) = \del Y-t_e{\left(Y^3-Y^2\right)}$$ we will use,
over a ground field $K_0\cong\Q\la t_0,t_1,\ldots\ra$
in which the elements $t_e$ form a differentially
independent set over $\Q$.  (That is, no tuple
$\la t_0,\ldots,t_n\ra$ is a zero of any nonzero
differential polynomial over $\Q$.)  Now $p_e(Y)$
certainly will have zeroes in the differential closure of $K_0$
(although its only zeroes in $K_0$ itself are the trivial ones $0$ and $1$),
but, imitating the proof of Corollary 6.3 in \cite{MMP06}
with $f(Y)=Y^3-Y^2$, we see that the set of its nonconstant zeroes
is algebraically independent over $\Q$.

Moreover, in an extension $K$
of $K_0$ such as we shall build, $p_e$ is constrainable
iff $K$ contains only finitely many zeros of $p_e$.
This condition is readily exploited to show $\Sigma^0_2$-completeness
of constrainability, which is to say, $\Pi^0_2$-completeness of $U_K$,
using requirements:
$$  \R_e:~~p_e\text{~has infinitely many zeroes in~}K \iff e\in\Inf.$$
These will show that the $\Pi^0_2$-complete set $\Inf=\set{e}{|W_e|=\infty}$
is $1$-reducible to $U_K$.  Since $U_K$ is already known
to be $\Pi^0_2$, it will therefore be $\Pi^0_2$-complete,
and many computable model theorists could fill in the rest
of these details immediately.

The differential field $K$ is built as a c.e.\ subfield of the differential
closure $\hat{K_0}$, in stages,
as a computably enumerable differential subfield of the (computable)
differential closure of $\Q\la t_0,t_1,\ldots\ra$.
Each requirement $\R_e$ will be eligible at infinitely many stages,
at each of which it will have the opportunity to add more elements
to $K$ if it needs to.

{\bf At stage $0$}, we apply Harrington's Theorem \ref{theorem:Harrington}
to (a computable presentation of) the differential field
$\Q\la t_0,t_1,\ldots\ra$, as described above,
with the $t_e$ all differentially independent over $\Q$.
This yields a differential Rabin embedding $f$
of $\Q\la t_0,t_1,\ldots\ra$ into a computable presentation $L$
of the differential closure of a c.e.\ subfield $K_0$
which is the image of $\Q\la t_0,t_1,\ldots\ra$
under $f$.  We may assume that the
function $e\mapsto t_e$ in the original presentation of
$\Q\la t_0,t_1,\ldots\ra$ was computable, and thus that
$e\mapsto f(t_e)$ is computable as well.  From here on,
we deal only with $K_0$, writing
$$a_e=f(t_e)\in K_0$$
for the image in $K_0$ of each differential transcendental $t_e$.
Within $L$, we can now name the nonconstant zeroes of the polynomials
$$p_e(Y)=\del Y-a_e{\left(Y^3-Y^2\right)}.$$  There must be infinitely many such zeroes,
since $L$ is differentially closed and $p_e$ has order $1$, and
each list $$\{ z_{e,0} < z_{e,1} < z_{e,2}<\ldots\}$$ of all
nonconstant zeroes of $p_e$ in $L$ is computable, uniformly in $e$.

{\bf We now move to stage $s+1$.}  Here we use the convention
that in our enumeration of all the c.e.\ sets $W_e$, for each $s$,
there is a unique pair $(e,x)$ for which $x\in W_{e,s+1}-W_{e,s}$
(that is, for which $x$ enters $W_e$ at stage $s+1$).  Of course,
$W_{e,0}=\emptyset$ for every $e$.  So, at the stage $s+1$,
we find the unique corresponding pair $(e,x)$, choose $i=|W_{e,s}|$,
and set
$$K_{s+1} = K_s(z_{e,i})\subset L.$$
Notice that since $$\del z_{e,i}=a_e{\left(z_{e,i}^3-z_{e,i}^2\right)},$$
this $K_{s+1}$ is closed under $\del$, hence is a differential subfield of $L$.
Lemma \ref{lemma:indep} will show that
$z_{e,i}$ is transcendental over $K_s$ as a field
(even though it is constrained over $K_s$ as a differential field),
making $K_{s+1}$ a purely transcendental field extension.

This completes the construction, and the differential field $K$
{\bf is the union of the fields} $K_s$ defined at each stage.
Thus $K$ is a c.e.\ subfield of $L$.  (A computable presentation of $K$
with domain $\omega$ may readily be found, since the infinite c.e.\ set $K$
is the image of $\omega$ under some injective total computable function $f$.
Just pull the differential field operations from $K$ back to $\omega$ via $f$.)

\begin{lemma}
\label{lemma:indep}
The set $\set{z_{e,i}}{e,i\in\omega}$ is algebraically independent
over $K_0$.
\end{lemma}
\begin{proof}
Each $z_{e,i}$ by itself is transcendental over $K_0$.
A theorem of Rosenlicht which appears as \cite[Theorem 6.2, p.\ 73]{MMP06},
along with Example 2 there,
shows that if $z=z_{e,i}$ and $\ztilde=z_{j,k}$ are algebraically dependent over $K_0$,
then $$a_e z^2=a_j\ztilde^2.$$  Applying $\del$ to both sides and using the equations
$$\del z=a_e{\left(z^3-z^2\right)}\quad \text{and}\quad \del\ztilde=a_j{\left(\ztilde^3-\ztilde^2\right)},$$
we derive a second equation
$$ z^2\del a_e + 2a_e^2{\left(z^4-z^3\right)} = \ztilde^2\del a_j + 2a_j^2{\left(\ztilde^4-\ztilde^3\right)}.$$
Applying $\del$ and substituting for $\del z$ and $\del\ztilde$ again yields a third polynomial relation
on $z$, $\ztilde$, $a_e$, $\del a_e$, $\del^2 a_e$, $a_j$, $\del a_j$, and $\del^2 a_j$.
But now these three distinct algebraic relations show that the field
$$\Q{\left(z,\ztilde,a_e,\del a_e,\del^2 a_e,a_j,\del a_j,\del^2 a_j\right)}$$
has transcendence degree at most $5$ over $\Q$.  Assuming that $e\neq j$,
this contradicts the differential independence of $a_e$ and $a_j$.
If $e=j$, then the equation $a_e z^2=a_j\ztilde^2$ actually showed
that $\ztilde=\pm z$, and the only way for $z$ and $-z$ both to be zeroes
of $p_e$ is for $z=0=-z$. 
Thus, every set $\{ z_{e,i},z_{j,k}\}$ of two distinct elements
is algebraically independent over $K_0$.

One then argues by induction.  Let $S$ be any subset of $\set{z_{e,i}}{e,i\in\omega}$
of size $n+1\geq 3$.
Form $$S'=S-\{ z_{e,i},z_{j,k}\}$$ by deleting any two
elements from $S$.  Then each of $z_{e,i}$ and $z_{j,k}$ is a transcendental
over $K_0(S')$, by inductive hypothesis, and the argument from the preceding paragraph,
with $K_0(S')$ and $\Q(S')$ in place of $K_0$ and $\Q$, shows that
$\{ z_{e,i},z_{j,k}\}$ is algebraically independent over $K_0(S')$, making $S$
algebraically independent over $K_0$.
\end{proof}

\begin{lemma}
\label{lemma:completeness}
In the construction above, the nonconstrainability set $U_K$
is $\Pi^0_2$-complete.
\end{lemma}
\begin{proof}
If $W_e$ is finite, then by our construction, there are only finitely many stages
at which any zero $z_{e,i}$ of the polynomial $p_e(Y)$ is added to $K$.
It follows from Lemma \ref{lemma:indep} that the only elements
$z_{e,i}$ in $K$ are the ones we specifically enumerated into $K_{s+1}$
at some stage $s+1$:  Lemma \ref{lemma:indep} shows that no finite set
of such elements $z_{e,i}$ generates any $z_{e',i'}$ as a field,
except those already in the finite set, and we remarker in the proof
that the field generated by this set is closed under differentiation.
Therefore, an easy induction on $s$ shows that, for every $e$ and $s$,
$p_e$ has exactly $|W_{e,s}|$ nontrivial zeroes in $K_s$.
But now, if $|W_e|=n$ is finite, then
$$(p_e(Y), Y(Y-1)(Y-z_{e,0})\cdot\ldots\cdot (Y-z_{e,n-1}))\notin T_K,$$
and so $p_e\notin U_K$.

Conversely, if $W_e$ is infinite, then $p_e(Y)$ has infinitely many zeroes in $K$,
since a new one is added at each stage at which
$W_e$ receives a new element.  But then there is no $q\in\KY$
with $(p_e,q)\notin T_K$:  such a $q$ would have to have order $0$,
since $p_e$ has order $1$,
and hence $q$ would have only finitely many zeroes
in $K$.  This would leave infinitely many $z_{e,i}$ in $K$ satisfying $(p_e,q)$,
yet clearly $z_{e,i}\not\cong_K z_{e,j}$ for all $i\neq j$ with
either $z_{e,i}$ or $z_{e,j}$ in $K$, since no isomorphism
fixing $K$ pointwise could map $z_{e,i}$ to any element except itself.
Thus $$p_e\in U_K\ \ \iff\ \ e\in\Inf,$$ and the computable injective
function $e\mapsto p_e$ is a $1$-reduction from $\Inf$ to $U_K$,
proving the lemma.  (Recall that Definition \ref{definition:constrainability}
shows that $U_K$ is a $\Pi^0_2$ set.)
\end{proof}
This completes the proof of Theorem \ref{theorem:Sigma2}.
\end{proof}

The usual examples of nonconstrainable polynomials
are along the lines of $p(Y)=\del Y$.  Based on just this, one
might wonder whether the nonconstrainable polynomials in $\KY$
are exactly those of the form $\del\ptilde$
for some $\ptilde\in\KY$.  In one direction, this holds true:
polynomials $\del\ptilde$ are nonconstrainable.
To see this, note that for every constant $c\in K$
and every $q\in\KY$ of order $<\ord (\del\ptilde)$
the polynomial $(\ptilde(Y)-c)$ would have order $>\ord(q)$,
so there would exist an $x_c\in\Khat$ satisfying
$$\ptilde(x_c)-c=0\neq q(x_c),$$ by Blum's axioms.
Every such $x_c$ would satisfy $(\del\ptilde,q)$,
yet for constants $c\neq d$, we would have
$$\ptilde(x_c)=c\neq d=\ptilde(x_d),$$ proving that
$(\del \ptilde,q)$ is not a constrained pair.

However, Theorem \ref{theorem:Sigma2} shows that nonconstrainability
is not always equivalent being the derivative of another polynomial.
\begin{corollary}
\label{cor:notSigma1}
There exist a computable differential field $K$ and a
differential polynomial $p\in\KY$ such that,
for all $\ptilde\in\KY$, $p\neq \del \ptilde$,
yet $p_0$ is not constrainable in $\KY$.
\end{corollary}
\begin{proof}
In the differential field $K$ built in Theorem \ref{theorem:Sigma2},
nonconstrainability cannot be equivalent to being
a derivative, since nonconstrainability cannot be defined by the
$\Sigma^0_1$ condition $$(\exists \ptilde\in\KY)\del \ptilde=p.$$
We argued above that every $\del \ptilde$ with $\ptilde\in\KY$
is nonconstrainable, so the opposite containment relation
must fail.  That is, there must exist some nonconstrainable
$p$ which is not a derivative.
\end{proof}

In fact, we can say specifically that in the field $K$ constructed
in Theorem \ref{theorem:Sigma2}, those $p_e$ which turned
out to be nonconstrainable are specific instances of Corollary \ref{cor:notSigma1}.
If such a $p_e$ were of the form $\del\ptilde$, then
$\ptilde$ would have appeared in $\KY$ at some finite stage $s$,
and thus $p_e$ would have been nonconstrainable
even in the differential subfield $K_s$ generated by
the elements enumerated by stage $s$.  But this did not happen:
each $p_e$ was constrainable within each $K_s$,
even though certain of them became nonconstrainable
in the larger field $K$.

Finally, we note that although $U_K$ is $\Pi^0_2$-complete for this
particular $K$, there may still be many other computable differential fields
$L$ for which $U_L$ is not this complex.  Theorem \ref{theorem:Sigma2} shows
that no definition of the constrainability set can be any simpler than
$\Pi^0_2$, but in specific cases a definition equivalent to constrainability
might have lower complexity.  In particular, for a constant differential
field $F$, and especially for the constant field $\Q$ itself, one suspects
that the complexity may be lower.  These questions remain open.

\section{Extension by Constrained Elements}
\label{sec:Kronecker}

Since the constraint set $T_K$ plays the same role for a differential
field $K$ that the splitting set $S_F$ plays for an algebraic field $F$,
it is natural to ask which results about $S_F$ carry over to $T_K$.
The most basic results about $S_F$ are those of
Theorem \ref{theorem:Kronecker}, given by Kronecker
in \cite{K1882}:  that $S_{\Q}$ is computable, and that
the splitting set $S_{F(x)}$ for a computable extension $F(x)$
of a computable field $F$ is computable from an oracle for $S_F$.

Moreover, the process of computing $S_{F(x)}$ from $S_F$
is uniform in $F$ and in one single piece of information about $x$,
namely, whether $x$ is algebraic or transcendental over $F$.
(It is often said that the process is uniform in the minimal polynomial
of $x$ over $F$, but if we know that $x$ is algebraic, then we can find
its minimal polynomial using the $S_F$ oracle.)

It follows that
every finitely generated computable field of characteristic $0$
has computable splitting set,
and this fact is extremely useful when one tries to build embeddings
or isomorphisms among computable fields.  The same holds in characteristic $p>0$,
since the splitting set of a finite field such as $\F_p$ is obviously computable.

In this section, we investigate the analogue of this result
for constrained differential field extensions.
If $K$ is a computable differential field, with
computable extension $K\la z\ra$, must $T_{K\la z\ra}\leq_T T_K$? We
will show that when $z$ is constrained over $K$, the answer
is always positive, uniformly in $z$, provided only that $K$ is nonconstant.
For constant $K$, the question remains open.

In the previous section we exploited the fact that if $K\subseteq L\subseteq\Khat$,
then $\Khat$ can fail to be (isomorphic over $L$ to) a differential closure of $L$, although
it must contain some subfield $\hat{L}\supseteq L$ which is a differential closure
of $L$.  The problem was that certain elements of $\Khat$ turned out not to be
constrainable over $L$.  When we consider a finitely generated differential field
extension $L=K\la z_0,\ldots,z_n\ra\subseteq\Khat$, this is no longer possible.
\begin{lemma}[{\cite[\S 4, Proposition~5]{KolchinConstrained}}]
\label{lemma:Kzhat}
Let $K$ be a differential field and $L$ a subfield  of any differential
closure $\Khat$ of $K$, with $L$ finitely differentially generated
as a differential field extension of $K$.  Then
$\Khat$ must also be a differential closure
of its differential subfield $L$.
\end{lemma}
Lemma \ref{lemma:Kzhat} shows that every $x\in\Khat$ satisfies
a constrained pair over $K\la z\ra$, but does not specify the two polynomials
which constitute that constrained pair.  Below we will prove Theorem \ref{theorem:fgalg},
stating that if we can recognize constrained pairs over $K$, then we can also recognize them
over $K\la z\ra$.  Once this is known, then given an oracle for $T_K$,
we will be able to identify the constrained pair over $K\la z\ra$ satisfied
by $x$, simply by searching for it.  However, there is a good deal
of work still to be done before we can prove Theorem \ref{theorem:fgalg}.

\begin{lemma}
\label{lemma:S_K}
For a computable differential field $K$ with Rabin embedding $g:K\to\Khat$,
and an arbitrary $z\in\Khat$, the splitting set $S_{K\la z\ra}$
is computable in an oracle for $T_K$, uniformly in $z$.
\end{lemma}
\begin{proof}
With the oracle $T_K$, we can find a constrained pair $(p,q)\in\TKbar$
satisfied by $z$.  Then, with $r=\ord(p)$, the differential field
$K\la z\ra$ is the field $$K{\left(z,\del z,\ldots \del^{r-1}z\right)}[\del^r z],$$
 and $p$ gives the minimal polynomial of $\del^r z$ over
the purely transcendental field extension  $K{\left(z,\del z,\ldots \del^{r-1}z\right)}$.
With this information, Kronecker's Theorem allows us to compute
the splitting set of $K\la z\ra$ from $S_K$.
\end{proof}
The same holds when $z$ is differentially transcendental over $K$:
then $K\la z\ra$ is just the field $K{\left(z,\del z,\del^2 z,\ldots\right)}$,
with computable transcendence basis $\set{\del^i z}{i\in\omega}$,
hence has a splitting algorithm.  However, the uniformity of the result
fails when we do not know whether $z$ is constrained or not.

In what follows, for $p \in \KY$, $[p]$ is the least differential
ideal in $\KY$ containing $p$. Let $r = \ord (p)$. Recall that the initial
$I_p$ of $p$ is defined as the leading coefficient
when $p$ is expressed as a polynomial in $\del^rY$ (say of degree $d$)
with coefficients in $K(Y,\ldots,\del^{r-1}Y)$, so that
$p = I_p (\delta^r Y)^d + \cdots$.
Let
\label{definition:hp}
$$h_p = I_p\cdot s_p,\text{~~~~where~}s_p = \frac{\partial p}{\partial
(\delta^r Y)}.$$ For example, if $p = Y(\del Y)^2 + \del Y+1$, then
$h_p = Y\cdot(2Y\del Y+1)$. For $J\subset \KY$, 
we define the \label{colonideal}\emph{colon ideal} $J:h_p^\infty$ by:
$$
J:h_p^\infty = \left\{f \in \KY\:|\: h_p^n\cdot f\in J \text{ for some } n \ge 0\right\}.
$$
It turns out that $[p]:h_p^\infty$  is a prime differential ideal if and only if
$(p):h_p^\infty$ is a prime ideal \cite[III.8, Lemma~6 and~IV.9, Lemma~2]{KolchinBook1973}.
Moreover, if $p$ is an irreducible polynomial, then $(p)$ is a prime ideal.
Therefore, in this case, $(p):h_p^\infty$ is prime as well. Thus, if $p$
is irreducible, then $[p]:h_p^\infty$ is a prime differential ideal.

\begin{lemma}
\label{lemma:Dima}
For every computable differential field $K$ and every differential
polynomial $p\in\KY$, the relation of equivalence mod $[p]:h_p^\infty$ is computable in $\KY$, uniformly
in $K$ and $p$, and so there is a computable presentation
of the differential ring $\KY\big/[p]:h_p^\infty$.
\end{lemma}
\begin{proof}
The basic point of the proof is that we have an algorithm,
called differential pseudo-division, which takes
an $f\in\KY$ as input and produces $g\in\KY$ of lower rank than $p$, with
$$f -g \in [p]:h_p^\infty,$$
(Recall that ``lower rank'' means that either
$\ord(g)<r=\ord(p)$ or else $\ord(g)=r$ and $\del^r
Y$ has strictly lesser degree in $g$ than its degree $d$ in $p$.)
First, if $\ord(f)=s>r$, then $\del^{s-r}p$ has order $s$ and is
linear in $\del^s Y$ with ``leading coefficient'' $s_p$.  So, by
subtracting an appropriate multiple of $\del^sY$ from $s_p\cdot f$,
we get an $f_0$ with order $<s$ such that $$(s_p\cdot f-f_0)\in
[p].$$  Repeating this process produces an $f_n$ of order $\leq r$.
If its order is exactly $r$, then we reduce it by multiples of $p$
until the degree of the variable $\del^rY$ in $f_m$ is $<d$
multiplying by $h_p$ whenever needed. A general algorithm is described in
\cite[I.9]{KolchinBook1973}.

Of course, if $S$ is the set of all $f\in\KY$
for which either $\ord(f)<r$ or else $\ord(f)=r$ and $\del^r Y$
has degree $<d$ in $f$,
then no two distinct elements of $S$ can be equivalent mod $[p]:h_p^\infty$,
since $$[p]:h_p^\infty\cap S=\{ 0\}.$$  So two elements of $\KY$
are equivalent mod $[p]:h_p^\infty$ iff the above algorithm
on those elements gives the same output for both.
Moreover, $S$ is computable and is the domain of our
computable presentation of $\KY/[p]:h_p^\infty$, with addition, multiplication,
and differentiation computed exactly as in $\KY$, with the result
reduced modulo $[p]:h_p^\infty$ to an element of $S$
in the sense of the algorithm described above.
\end{proof}
\begin{corollary}
\label{cor:K<z>}
For every computable differential field $K$, every Rabin
embedding $g:K\to\Khat$, and every $z\in\Khat$,
the differential subfield $g(K)\la z\ra$ is computably presentable,
uniformly in the minimal differential polynomial $p$ of $z$ over $K$.
(In turn, $p$ can be computed from $z$ using a $D_K$ oracle.)
\end{corollary}
\begin{proof}
$K\la z\ra$ is isomorphic to the differential field $\Quot(\KY/[p]:h_p^\infty)$,
with $z$ mapping to the equivalence class of $Y$.
In this case, the irreducible polynomial $p \in \KY$ forms a characteristic set of
the prime differential ideal $P := \ker (\KY \to K\la z \ra)$ \cite[I.10]{KolchinBook1973}.
\end{proof}
Of course, this corollary is clear for other reasons,
even without needing to know $p$.
$g(K)\la z\ra$ is a c.e.\ differential subfield of $\Khat$,
hence can easily be pulled back to a computable presentation.
The interesting point is that one cannot readily use the
subfield $g(K)\la z\ra$ to prove the uniformity
over $p$ in Lemma \ref{lemma:Dima}.
It is easy to find zeroes $z\in\Khat$ of the given $p$,
but unless $D_K$ is computable, one cannot decide
which of these zeroes, if any, is generic for $p$ -- that is,
which $z$ have minimal differential polynomial $p$.

Moreover, if $p$ were unconstrainable, then no $z$
at all would be generic for $p$ over $K$.  So the proof of
Lemma \ref{lemma:Dima} is essential to the study
of computable differential fields, particularly those for
which $D_K$ or $U_K$ is noncomputable.  (For computable fields $F$,
the analogous procedure works for all irreducible $p\in F[X]$,
so one only needs to be able to decide the splitting set $S_F$,
which is always c.e., whereas in general $U_K$ is only $\Sigma^0_2$.)

We will need the following lemma as well.
\begin{lemma}
\label{lemma:Delta1}
Let $\Khat$ be a computable differentially closed field.
Then it is computable, given arbitrary $m$, $n$, and coefficients for
differential polynomials $f_0,\ldots,f_m$ and $g_0,\ldots,g_n$
in $\Khat\{ Y\}$, whether $V(f_0,\ldots,f_m)\subseteq V(g_0,\ldots,g_n)$,
where these are the differential varieties in $\Khat$ defined by these polynomials.
Consequently, it is also uniformly computable whether
$V{\left(\fbar\right)}=V(\gbar)$, and whether containment or
equality holds between the radical differential ideals
$\sqrt{{\left[\fbar\right]}}$ and $\sqrt{[\gbar]}$.
\end{lemma}
\begin{proof}
This follows from quantifier elimination for the theory $\textbf{DCF}_0$, but it can
also be seen as a pleasing application of the Differential Nullstellensatz,
which states that with $\Khat$ 
differentially closed, $V{\left(\fbar\right)}\subseteq V(\gbar)$
iff $\sqrt{[\gbar]}\subseteq\sqrt{{\left[\fbar\right]}}$.  (The analogous result over
algebraically closed fields is Hilbert's Nullstellensatz.)
So the statement about radical differential ideals will follow from
that about varieties.  

Moreover, the reversal between the two inclusions
quickly yields the lemma, since
$V(\fbar)\subseteq V(\gbar)$ is a universal condition, saying that
every common zero of all $f_i$ is also a zero of every $g_j$, while
$\sqrt{[\gbar]}\subseteq\sqrt{{\left[\fbar\right]}}$ is an existential
condition:  each of the (finitely many) $g_i$ in $\gbar$ lies
in $\sqrt{{\left[\fbar\right]}}$ iff there exist $k_j$ and a linear combination
(over $\Khat\{ Y\}$) of the $f_i$'s and their derivatives
such that the combination is equal to $g_j^{k_j}$.
Being defined uniformly by both an existential and a universal
condition, the containment is therefore computable.
\end{proof}

We now come to the differential analogue of Kronecker's Theorem,
for finitely generated constrained extensions.

\begin{theorem}
\label{theorem:fgalg}
Let $L$ be any computable ordinary differential field
with nontrivial derivation $\del$, and $K$ its image under
any differential Rabin embedding of $L$ into any $\Khat$.
Then for every $z\in\Khat$, the constraint set $\overline{T_{K\la z\ra}}$
is computable in an oracle for $T_K$.  Also, the constrainability set
$\overline{U_{K\la z\ra}}$ is computably enumerable relative
to $T_K$.  Both the computation of $T_{K\la z\ra}$
and the enumeration of $\overline{U_{K\la z\ra}}$
are uniform in $z$ and $T_K$.
\end{theorem}
\begin{proof}
First, using the oracle $T_K$, search for a pair $(p_z,q_z)\in\TKbar$
satisfied by $z$.  Since $z\in\Khat$, we eventually find such a pair,
in which $p_z$ is the minimal differential polynomial of $z$ over $K$.
Set $r_z=\ord(p_z)$.  Then $K\la z\ra$ has transcendence basis
$$B_0={\left\{ z,\del z,\ldots,\del^{r_z-1}z\right\}}$$ over $K$, and is generated by this basis
along with $\del^{r_z}z$, whose minimal polynomial over the transcendence
basis is $p_z{\left(z,\ldots,\del^{r_z-1}z,Y\right)}$.  With this information, we can compute
the splitting set $S_{K\la z\ra}$ from $S_K$, which in turn is computable
from our $T_K$-oracle, by Theorem \ref{theorem:constraints}.

Moreover, $T_K$ can also compute $D_K$, the set of finite subsets
of $\Khat$ which are algebraically dependent over $K$.  But a finite
subset $S$ of $\widehat{K\la z\ra}=\Khat$ is algebraically independent over $K\la z\ra$
iff $S\cap B_0 =\emptyset$ and $S\cup B_0\notin D_K$.
Thus $$D_{K\la z\ra}\leq_T D_K\leq_T T_K,$$ again uniformly in $z$.
(``Computing $D_{K\la z\ra}$'' normally means deciding algebraic
independence over $K\la z\ra$ for all finite subsets from
some particular differential closure of $K\la z\ra$.
In this case, by Lemma \ref{lemma:Kzhat},
the identity map from $K\la z\ra$ into $\Khat$
is a differential Rabin embedding, and we have
given a process for deciding algebraic independence over $K\la z\ra$
which works for all finite tuples from $\Khat$.)

Since $K\la z\ra$ is a computable differential field, $T_{K\la z\ra}$
is $\Sigma^0_1$, uniformly in $z$ (and even in $K$).  So we need
only show that its complement is $\Sigma^0_1$ relative to $T_K$.
Briefly summarizing, the following argument says that there are two ways
for $(p,q)$ to lie in $T_{K\la z\ra}$: either $p$ is nonconstrainable there,
or else $q$ is not the correct constraint for $p$.  Constrainability is $\Sigma_1$
in an oracle for $D_K$, which we have.  Assuming $p$ is constrainable,
we use the differential closure to find some $x$ satisfying $(p,q)$,
and then use Lemma \ref{lemma:Dima} to compare $q$ over $[p]:h_p^\infty$
with a constrained pair $(p,\qtilde)$ over $K\la z\ra$, which $x$ is known to satisfy.
(Possibly $(p,\qtilde)$ is not actually a constrained pair over $K\la z\ra$,
but it is known to generate a principal type there, which is sufficient
for our purpose.)

We claim that, for every $(p,q)\in (K\la z\ra\{Y\})^2$, the following process
will halt iff $(p,q)\notin T_{K\la z\ra}$.
Set $r= \ord(p)$, and check first that $r>\ord(q)$, and that
$p$ is monic and irreducible as an algebraic polynomial
over $K\la z\ra$.  (This uses $S_{K\la z\ra}$.)
If so, then we search for an $x\in\Khat$ with the properties that
$$p(x)=0\neq q(x)\quad\text{and}\quad{\left\{ x,\del x,\ldots,\del^{r-1}x\right\}}\notin D_{K\la z\ra}.$$
Assuming $(p,q)\notin T_{K\la z\ra}$, such an $x$ exists, so eventually
we find one, and Fact \ref{fact:order} shows that $p$ is the minimal
differential polynomial of this $x$ over $K\la z\ra$, hence is constrainable,
by Fact \ref{fact:mindiffpoly}.
(If $p$ were not constrainable, then no such $x$ could exist, since we
have already confirmed that $p$ is irreducible.)  
Moreover, the differential ideal $[p]:h_p^\infty$ in
$K\la z\ra\{ X\}$ is precisely the ideal $$\set{j\in K\la z\ra\{ X\}}{j(x)=0}.$$
Having found this $x$, we search for some $u$ in the c.e.\ subfield $K\la x,z\ra$
of $\Khat$ such that $x$ and $z$ both lie in $K\la u\ra$.

Theorem \ref{theorem:Kolchin} shows that we do eventually find such a $u$ in $\Khat$,
since we assumed $\del$ to be nontrivial, and indeed we find
differential rational functions $f,g,h$ with coefficients in $K$
such that $$f(x,z)=u,\quad g(u)=x,\quad \text{and}\quad h(u)=z.$$
We also use the oracle to find a pair $(p_u,q_u)\in\TKbar$ satisfied by $u$,
and another pair $(p_x,q_x)\in\TKbar$ satisfied by $x$.
Now $$h(f(x,z))-z=0,$$
and if we choose $h_0,h_1\in K\la z\ra\{ X\}$ so that
$$\frac{h_0(X)}{h_1(X)}=h(f(X,z))-z,$$
then $h_0(x)=0\neq h_1(x)$.
Similarly, setting $$\frac{g_0(X)}{g_1(X)}=g(f(X,z))-X\quad \text{and}\quad \frac{p_0(X)}{p_1(X)}=p_u(f(X,z))$$ gives
$g_0(x)=0\neq g_1(x)$ and $p_0(x)=0\neq p_1(x)$,
with all these polynomials in $K\la z\ra\{ X\}$.

Now we define $\qtilde\in K\la z\ra\{ X\}$ by
$$\qtilde(X)=q_u(f(X,z))\cdot 
g_1(X)\cdot h_1(X)\cdot
p_1(X)\cdot h_p(X),$$
with $h_p$ as on page~\pageref{definition:hp} for our $p$.
Suppose $\xtilde\in\Khat$ satisfies $(p,\qtilde)$.
Now every $j\in K\la z\ra\{ X\}$ with $j(x)=0$ lies in $[p]:h_p^\infty$,
hence must have $j(\xtilde)=0$ (since $h_p(\xtilde)\neq 0$).  
We also know that $x$ satisfies $(p,\qtilde)$, since $h_p$
has strictly lesser rank than the minimal differential polynomial
$p$ of $x$, and thus must have $h_p(x)\neq 0$.
We set $$\utilde=f(\xtilde,z),$$ so by our assumption about
$\xtilde$, we know $q_u(\utilde)\neq 0$.
Also, $$0=p_u(u)=p_u(f(x,z)),$$ so $p_0$
lies in $[p]:h_p^\infty$, while $p_1(\xtilde)\neq 0$, and
therefore $$0= p_u(f(\xtilde,z))=p_u(\utilde)$$ as well.
Therefore this $\utilde$ satisfies the constrained pair
$(p_u,q_u)\in\TKbar$, and so the map $\sigma$
with $\sigma(u)=\utilde$ which restricts to the identity on $K$
is an isomorphism from $K\la u\ra$ onto $K\la\utilde\ra$.

Now $0=h_0(\xtilde)=g_0(\xtilde)$, since these differential polynomials
both have $x$ as a zero, while we know $$g_1(\xtilde)\cdot h_1(\xtilde)\neq 0.$$
It follows that $$h(f(\xtilde,z))-z=0=g(f(\xtilde,z))-\xtilde,$$
and so $h(\utilde)=z$ and $g(\utilde)=\xtilde$.
But $g$ and $h$ have coefficients in $K$, so
$$\sigma(z)=\sigma(h(u))=h(\sigma(u))=h(\utilde) = z~~\text{and}~~
\sigma(x)=\sigma(g(u))=g(\sigma(u))=g(\utilde) = \xtilde.$$
Thus this $\sigma$ maps $x$ to $\xtilde$, fixes $z$, and fixes $K$ pointwise,
so $\sigma$ witnesses that $x\cong_{K\la z\ra}\xtilde$.
We have thus shown that $(p(X),\qtilde(X))$
generates a principal type over $K\la z\ra$ and is satisfied by $x$.

Notice that $(p,\qtilde)$ could fail to be a constrained pair, since $\qtilde$
could have order $\geq\ord(p)$, but it has the basic property
of constrained pairs, which is to generate a principal type.
(Recall that this means that every two elements of $\widehat{K\la z\ra}$
which both satisfy $(p,q)$ are isomorphic over $K\la z\ra$.)

Now suppose $t\in V(p,q)$.  Then $q(t)=0\neq q(x)$,
so $$x\not\cong_{K\la z\ra} t.$$  Since $x$ satisfies $(p,\qtilde)$,
$t$ cannot satisfy it, and with $p(t)=0$, this forces $\qtilde(t)=0$.
Thus $$V(p,q)\subseteq V(p,\qtilde).$$
We now claim that $$V(p,\qtilde) = V(p,q)\quad \iff\quad (p,q)\in \overline{T_{K\la z\ra}}.$$
For the forwards direction, suppose $t$ and $v$ both satisfy $(p,q)$.
Then $$t\notin V(p,q)=V(p,\qtilde),$$ but $p(t)=0$, so $\qtilde(t)\neq 0$.
Thus $t$ satisfies $(p,\qtilde)$, and so does $v$, by the same argument.
Since $(p,\qtilde)$ generates a principal type, we have $t\cong_{K\la z\ra} v$,
which proves $$(p,q)\in\overline{T_{K\la z\ra}}.$$  For the converse,
suppose $$V(p,\qtilde)\neq V(p,q).$$  Then there must exist some
$w\in V(p,\qtilde)-V(p,q)$, since we saw above that $V(p,q)\subseteq V(p,\qtilde)$.
Thus $p(w)=0$ and $q(w)\neq 0$, so $w$ satisfies $(p,q)$, as does $x$.
Yet $$\qtilde(x)\neq 0 =\qtilde(w)$$ (since $w\in V(p,\qtilde)$), and so
$x\not\cong_{K\la z\ra} w$, which proves $$(p,q)\notin\overline{T_{K\la z\ra}}.$$

So we use Lemma \ref{lemma:Delta1} to check whether $V(p,\qtilde)=V(p,q)$,
and if so, we conclude that $(p,q)$ is a constrained pair over $K\la z\ra$.
Thus, being such a constrained pair is $\Sigma^0_1$ relative to $T_K$,
and not being such a constrained pair is $\Sigma^0_1$ without any oracle.
Finally, since a $T_{K\la z\ra}$-oracle can enumerate $\overline{U_{K\la z\ra}}$,
so can a $T_K$-oracle, now that we know $T_{K\la z\ra}\leq_T T_K$.
This completes the proof of Theorem \ref{theorem:fgalg}.
\end{proof}

\section{Extension by Differential Transcendentals}
\label{sec:trans}

\begin{lemma}
\label{lemma:irred}
Suppose that $p\in\KY$ is algebraically irreducible and has order $r$.
If an element $y\in\Khat$ has $p(y)=0$, but $p$ is not the minimal
differential polynomial of $y$, then there exists
some $h\in\KY$ of order $<r$ such that $h(y)=0$ as well.
\end{lemma}
\begin{proof}
This $y$ must have a minimal differential polynomial $h$ over $K$.
If $h$ has order $\geq r$, then by minimality it must have order $r$
and degree $< d$, the degree of $\del^rY$ in $p(Y)$.
But $h(y)=0$ gives the minimal (algebraic) polynomial in
$$K\big(y,\del y,\ldots,\del^{r-1}y\big)[\del^rY]$$
of $\del^ry$ over $K\big(y,\del y,\ldots,\del^{r-1}y\big)$.
Thus $h$ must divide $p$, contradicting the irreducibility of $p$.
\end{proof}

\begin{lemma}
\label{lemma:counterEX}
Suppose that $p\in K\la z\ra\{ Y\}$, where $z$ is differentially transcendental over $K$.
Suppose further that all coefficients of $p$ lie
within the field $K(z,\del z,\ldots,\del^r z)$.
If there are elements $x,y$ in the differential closure of $K\la z\ra$
for which $p(x)=p(y)=0$ but $x\not\cong_{K\la z\ra} y$, then there
exists $g\in K(z,\ldots,\del^r z)\{ Y\}$ of strictly lesser rank
(in $z$) than $p$, such that
$$
\text{either}\ \
g(x)=0\neq g(y)\ \  \text{or}\ \  g(y)=0\neq g(x).$$
If $p$ is constrainable over $K\la z\ra$,
and $g(x)=0\neq g(y)$, then there is also some $\ytilde$ and some
$\gtilde\in K\big(z,\ldots,\del^{r-1} z\big)\{ Y\}$ such that
$\gtilde(x)=p(\ytilde)=0\neq \gtilde(\ytilde)$.
\end{lemma}
\begin{proof}
The proof of the main statement is by induction on the rank of $p$ with respect to
the variable $z$, where we view $p$ as a polynomial in both
$Y$ and $z$, clearing denominators if needed.
(Technically, this is a transfinite induction, since ranks
of differential polynomials, even in our simple definition,
form an order of type $\omega^2$.)
The base case is trivial, since when $r=-1$,
then the two elements $x$ and $y$ must satisfy $x\not\cong_K y$,
and Proposition \ref{prop:noniso} then provides a polynomial
$g\in\KY$ with $g(x)=0\neq g(y)$ (or vice versa).

For the inductive step, let $p$, $x$, and $y$ be as described.  Since
$x\not\cong_{K\la z\ra} y$, Proposition \ref{prop:noniso}
yields some $h\in K\la z\ra\{ Y\}$ with $$h(x)=0\neq h(y).$$
That proposition does not provide
any \textit{a priori} bound on the order of $z$
in the coefficients of this $h$, but we may view
$h(Y)$ as a differential polynomial in the two variables $Y$ and $z$
(clearing denominators in $h(Y)$ if needed) and reduce $h$
modulo $[p]:h_p^\infty$ with respect to $z$, finding some $g$,
whose rank with respect to $z$ is less than the rank of $p(Y)$,
and some exponent $k$ such that $h_p^k\cdot h-g\in [p]$ (see \cite[Section~I.9]{KolchinBook1973}).
Therefore,
$$h_p(x)^k\cdot h(x)-g(x)=0,$$ forcing
$g(x)=0$. Now  either $g$ satisfies the lemma
(if $g(y)\neq 0$), or we apply induction to $g$ (if $g(y)=0$),
since $g$ also has strictly lesser rank than $p$.

Now suppose that $p$ is also constrainable.
(Notice that, if $(p,q)\in \overline{T_{K\la z\ra}}$,
then $q$ has lower order than $p$ in the variable $Y$,
but not necessarily in $z$.  So we cannot simply take $\gtilde$ to equal $q$.)
Assume that $g$ has positive degree in $\del^rz$,
since otherwise we can set $\gtilde=g$.
This degree must be less than that of $p$, since $g$ has lower
rank than $p$, and so we may apply the pseudo-division
algorithm from Lemma \ref{lemma:Dima} repeatedly,
with respect to the degree of the variable $\del^rz$ over
coefficients from $K(z,\ldots,\del^{r-1} z)\{ Y\}$.  This yields
\begin{align*}
i_g(Y)^{m_g}\cdot p(Y) &= g(Y)\cdot d_1(Y) + q_1(Y)\\
i_{q_1}(Y)^{m_1}\cdot g(Y) &= q_1(Y)\cdot d_2(Y) + q_2(Y)\\
&~~~\vdots\\
i_{q_n}(Y)^{m_n}\cdot q_{n-1}(Y) &= q_n(Y)\cdot d_{n+1}(Y)+q_{n+1}(Y),
\end{align*}
where, for all $i\leq n$, $q_{i+1}(Y)$ has strictly
lower degree in $\del^rz$ than $q_i(Y)$ has (including $i=0$, with $q_0=g$).
The process ends when we reach an $n$ for which
$q_{n+1}(Y)$ has degree $0$ in $\del^rz$.
Notice that if $q_{n+1}(Y)$ were the zero polynomial,
then $q_n(Y)$ would have positive degree in $\del^rz$ and would
divide $$i_g(Y)^{m_g}\cdot i_{q_0}(Y)^{m_0}\cdot\ldots\cdot i_{q_n}(Y)^{m_n}\cdot p(Y),$$
which is impossible,
since $$i_g(Y)^{m_g}\cdot i_{q_0}(Y)^{m_0}\cdot\ldots\cdot i_{q_n}(Y)^{m_n}$$
is free of $\delta^r z$ and
the constrainable polynomial $p$, even after its
denominators were cleared, must be irreducible
(as a polynomial in $Y,\del Y,\del^2 Y\ldots$)
and can have no nontrivial factor from $K\la z\ra$.  Therefore,
$$q_{n+1}(Y)\in K\big(z,\ldots,\del^{r-1} z\big)\{ Y\}$$ is nonzero.
But by induction on $i$, we see that
either $i_{q_i}(x) = 0$, or else $q_i(x)=0$ for every $x$,
since $p(x)=g(x)=0$.  In the first case, we take $\tilde g := i_g$,
while in the second case, $\gtilde=q_{n+1}$ is the desired
differential polynomial.  In both cases, we get  $$\gtilde(x)=0=p(x),$$ while
Fact \ref{fact:mindiffpoly} yields some $\ytilde$ in the differential closure of
$K\la z\ra$ with $p(\ytilde)=0\neq\gtilde(\ytilde)$.
\end{proof}

We state the next lemma in a specific form which will be useful
for the results in this section.
\begin{lemma}
\label{lemma:constraintvariety}
For any differential polynomials $p,q,h\in\KY$,
the set $$\set{x\in\Khat}{p(x)=h(x)=0\neq q(x)}$$ is empty
iff $1\in \sqrt{[p,h]}:q$.
\end{lemma}
\begin{proof}
If $p(x)=h(x)=0\neq q(x)$, then every $f\in\sqrt{[p,h]}:q$ has
$f^n\cdot q^n\in [p,h]$ for some $n$, hence has $f(x)=0$,
precluding the constant function $1$ from appearing in $\sqrt{[p,h]}:q$.

Conversely, if $1\notin \sqrt{[p,h]}:q$, then, for every $n$, we have $q^n\notin [p,h]$,
and the Differential Nullstellensatz yields an $x\in\Khat$
with $p(x)=h(x)=0$ but $q(x)\neq 0$.
\end{proof}
\begin{corollary}
\label{cor:constraintvariety}
Suppose the pair $(p,q)$ from $K\{ Y\}$
has $p$ monic and irreducible of order $>\ord(q)$.
Then the following are equivalent.
\begin{enumerate}
\item\label{it:1041}
$(p,q)\in\TKbar$.
\item\label{it:1042}
For all $h\in K\{ Y\}$ of lesser rank than $p$, $1\in\sqrt{[p,h]}:q$.
\item\label{it:1043}
For all $g\in K\{ Y\}$, we have either $1\in\sqrt{[p,g]}:q$
or $g\in \sqrt{[p]}:q$.
\end{enumerate}
\end{corollary}
\begin{proof}
If $(p,q)$ is a constrained pair, then $p$ is the minimal differential
polynomial of some $y\in\Khat$, and so $q(y)\neq 0\neq h(y)$
for every $h$ of lesser rank than $p$.  But now every $x$
with $p(x)=0\neq q(x)$ has $x\cong_K y$, so $h(x)\neq 0$
for all such $h$, forcing $$\set{x\in\Khat}{p(x)=h(x)=0\neq q(x)}$$
to be empty.  By  Lemma \ref{lemma:constraintvariety},
it follows that $1\in \sqrt{[p,h]}:q$ for all such $h$.

Conversely, suppose $1\in\sqrt{[p,h]}:q$
for all $h\in K\{ Y\}$ of lesser rank than $p$.
Then, for every such $h$,
$$\set{x\in\Khat}{p(x)=h(x)=0\neq q(x)}$$ is empty, by Lemma \ref{lemma:constraintvariety}.
Therefore, if $x,y\in\Khat$ both satisfy $(p,q)$,
then $h(x)\neq 0\neq h(y)$ for every such $h$,
and Lemma \ref{lemma:counterEX} then shows
that $x\cong_K y$.  Thus, $(p,q)\notin T_K$, and so~\eqref{it:1041} $\iff$~\eqref{it:1042}.

The equivalence of~\eqref{it:1042} and~\eqref{it:1043} follows because every $g$
is equivalent modulo $\sqrt{[p]}:q$ to some $h$ of lesser rank than $p$.
If this $h$ is zero, then $g\in\sqrt{[p]}:q$, while if not, then
every $x$ with $p(x)=0\neq q(x)$ has $g(x)=h(x)\neq 0$,
and so $1\in \sqrt{[p,g]}:q$ by Lemma \ref{lemma:constraintvariety}.
\end{proof}

\begin{lemma}\label{lem:difoveralgclosed} If $\Khat$ is not algebraic over $K$, then this extension is of infinite transcendence degree, and indeed $\Khat$ contains elements
of arbitrarily large order over $K$.
\end{lemma}
\begin{proof}
The argument showing the first part of the statement is due to Michael Singer. For simplicity, assume that $K$ is algebraically closed. When $K$ and $\Khat$ have the same field of constants, one can use \cite[Corollary, p.~489]{Rosenlicht} to show that the transcendence degree of $\Khat$ over $K$ is infinite.
To see this, assume that this transcendence degree is finite, say $n$. Since $K$ is algebraically closed, $n \geqslant 1$. Let $y$ be in $\Khat$ and transcendental over $K$. Let
$$v_i = y^i,\quad i= 1,\ldots,n+1.$$ Since $\Khat$ is differentially closed, let $u_i \in \Khat$ satisfy $$\delta(u_i)/u_i = - \delta(v_i)\quad i = 1,\ldots, n+1.$$ The conclusion of the corollary implies that 
there exist constants $c_i$ such that $\sum_i c_iy^i$ is algebraic over $K$, which is a contradiction.

When $\Khat$ has new constants, argue as follows. Let $C$ be the constants of $\Khat$. If $\Khat$ is not algebraic over $KC$, then there is $y\in\Khat$ transcendental over $KC$, and we can argue as above to get a contradiction. �Therefore, we can assume $\Khat$ is algebraic over $KC$. We can assume that $K$ contains a non-constant $x$ (otherwise there are lots of elements in $\Khat$ that are algebraically independent over $K$, for example, non-zero solutions of $$\delta(Y) = z^n\cdot Y,\quad n \ge 0,$$ where $z \in \Khat$ with $\delta(z)=1$)  and that $\delta(x) = 1$ (replacing $\delta$ by $(1/\delta(x))\cdot\delta$ if needed). Since $K$ is algebraically closed, there is a constant $c \in \Khat$ not algebraic over $K$. Since $\Khat$ is differentially closed, there is $y\in \Khat$ such that $$\delta(y)=1/(c+x).$$ Since $y$ is algebraic over $KC$, taking traces, we see  \cite[Exercise~1.24]{Michael} that there is $z\in KC$ such that $$\delta(z)=1/(c+x).$$ Since $KC$ is an extension of $K(c)$ by constants, \cite[Proposition~1.2]{Risch} yields $w=p(c)/q(c) \in K(c)$ such that $$\delta(w) = 1/(c+x).$$ Expanding $w$ in partial fractions with respect to $c$, differentiating and comparing terms shows that this is impossible.

Now we consider the orders of elements of $\Khat$.
If $K$ is nonconstant, Theorem
\ref{theorem:Kolchin} then implies that $\Khat$ contains
individual elements of arbitrarily large order over $K$.
In the case where $K$ is a constant differential field,
let $x \in \Khat$ be such that $x'=1$; such an element
exists in $\Khat$ by Blum's axioms, and must be transcendental
over $K$ since an element algebraic over the constant field
$K$ would also be constant. Then, for all $a_1,...,a_n \in K$,
$$
f(a_1,...,a_n) := x^n/n! + a_1x^{n-1} + ... + a_n
$$
satisfies
$f^{(n)}=1$.
Now there exist $a_1,...,a_n \in K$ such that $f(a_1,...,a_n)$ is not a solution
of any polynomial differential equation $F$ of order $n-1$ or less. Indeed, for each $q < n$,
the coefficient of $x^0$ in $f^{(q)}$ is equal to $$q!\cdot a_{n-q}.$$
(This coefficient is often referred to as the ``constant term'' of $f^{(q)}$,
but it would be confusing to call it that in this context.)
Hence, for all $q_1,\ldots,q_r$ from $\{ 0,1,\ldots,n-1\}$
and all $n_1,\ldots,n_r$, the coefficient of $x^0$ of
$$
\prod\nolimits_i \big(f^{(q_i)}\big)^{n_i}$$
is equal to $$\prod\nolimits_i \big(q_i!\cdot a_{n-q_i}\big)^{n_i}
$$
Thus, the coefficient of $x^0$ in $F(f)$ is a polynomial in $a_1,...,a_n$, which
is not identically zero. This implies the result as $K$ is infinite.
\end{proof}

In the proof of Theorem~\ref{theorem:fgtrans}, we use the technique of characteristic sets from differential algebra. We will give a short introduction to it for the convenience of the reader unfamiliar with characteristic sets and differential rankings. Earlier, we introduced the ring of differential polynomials in one differential indeterminate, which we denoted by $K\{Y\}$ most often. Let us now have finitely many differential indeterminates $\{y_1,\ldots,y_n\}$ (in the proof of Theorem~\ref{theorem:fgtrans}, we will just need the case of $n=2$). 
A {\it ranking} 
%\cite{Kol} 
is a total order $>$ on the set of derivatives $$D := \left\{y_i^{(p)}\:\big|\: p\geq 0,\ 1\leq i\leq n\right\}$$
satisfying the following conditions for all $p > 0$ and
$u,v\in D$:
\begin{enumerate}
\item $u^{(p)} > u$, and
\item $u \geq v \Longrightarrow u^{(p)} \geq v^{(p)}$.
\end{enumerate}
Similarly to our previous introduced notation, $K\{y_1,\ldots,y_n\} := K[D]$ (as commutative rings, and the $\delta$-action is defined naturally) and, for $u = y_j^{(q)}$, $\ord u:=q$. If $f \in K\{y_1,\dots,y_n\}\setminus K$, then $\ord f$ denotes the maximal order of
the derivatives appearing effectively in $f$.

A ranking $>$ is called
{\it orderly} if, for all $u$, $v \in D$, $\ord u > \ord v$ implies $u > v$. A ranking $>$ is called {\it elimination} with $y_{i_1} < \ldots < y_{i_n}$ if $$y_{i_j}^{(p)} < y_{i_{j+1}}^{(q)},\quad p,\: q \geq 0,\quad 1\leq j \leq n.$$ 
For example, in an orderly ranking with $y_1 < y_2$, we have: $y_1' > y_2$ and $y_1' < y_2'$. However, in an elimination ranking with $y_1 < y_2$, we have: $y_1' < y_2$ and $y_1' < y_2'$.

Let a ranking $>$  be fixed on $S$.
The derivative $y_j^{(p)}$ of the highest rank appearing
in  $f \in K\{y_1,\dots,y_n\} \setminus K$
is called the {\it leader} of $f$, which we denote by $u_f$.
%The indeterminate $y_j$ is called the {\it leading variable} of $f$ and
%denoted by $\lv f.$
Represent $f$ as a univariate polynomial in $u_f$:
\begin{equation}\label{eq:leader}
f = i_f u_f^d + a_1 u_f^{d-1} + \ldots + a_d.
\end{equation}
The monomial $u_f^d$ is called the {\it rank} of $f$. % and is denoted by $\rk f.$
Extend the ranking relation on derivatives  to ranks:
$u_1^{d_1}>u_2^{d_2}$ if either $u_1>u_2$, or
$u_1=u_2$ and $d_1>d_2$.
As at the beginning of Section~\ref{sec:Kronecker},
the polynomial $i_f$ is called the {\it initial} of $f$.
Applying $\delta$ to $f$, we obtain
$$ \delta f = \frac{\partial f}{\partial u_f}\delta u_f + \delta i_f u_f^d +
\delta a_1 u_f^{d-1}+\ldots + \delta a_d.
$$
The leader of $\delta f$ is $\delta u_f$, and the initial
of $\delta f$ is called the {\it separant} of $f$, which is denoted by $s_f$.
Note that the initial of any
proper derivative (i.e., a derivative of order greater than $0$) of $f$ is equal to $s_f$. 
For example, if $K=\Q(x)$, $\delta(x) =1$, the ranking is elimination with $y_1 > y_2$ and $$f = (y_2y_1'+1)y_1''^2+ y_1'^2y_2'^3y_1''+y_1y_2^3+x,$$ then $u_f =y_1''$, the rank of $f$ is $y_1''^2$, $i_f = y_1'y_2+1$,
\begin{align*}
\delta f = &\left(2(y_2y_1'+1)y_1''+y_1'^2y_2'^3\right)y_1'''+\\
&+(y_2y_1''+y_2'y_1')y_1''^2 + (2y_1'y_1''y_2'^3+3y_1'^2y_2'^2y_2'')y_1''+y_1'y_2^3+3y_1y_2^2y_2'+1
\end{align*}
and $s_f = 2(y_2y_1'+1)y_1''+{y_1'}^2{y_2'}^3$.

We say that  $f\in  K\{y_1,\dots,y_n\}$ is {\it partially reduced}
w.r.t. $g \in K\{y_1,\dots,y_n\}\setminus K$ if no proper derivative of $u_g$ appears in $f$.
Also, $f$ is said to be {\it algebraically reduced} w.r.t.
$g$ if $\Deg_{u_g} f<\Deg_{u_g}g$.
Finally, $f$ is called {\it reduced} w.r.t.  $g$ if $f$ is partially and algebraically reduced
 w.r.t. $g$. Let
$A \subset K\{y_1,\ldots,y_n\}\setminus K$. 
For example, $y_1'+y_2$ is reduced w.r.t. $y_2'+y_1$ in an orderly ranking but is not partially reduced w.r.t. $y_2'+y_1$ in an elimination ranking with $y_1 > y_2$. Also, $y_1$ is partially reduced but not reduced w.r.t. $y_1$.
We say that
$A$ is {\it autoreduced} 
if each element of $A$ is reduced
w.r.t. all the others. So, in an orderly ranking, the set $y_1'+y_2, y_2'+y_1$ is autoreduced. However, it is not autoreduced in any elimination ranking.

Every autoreduced set is finite
\cite[Chapter I, Section 9]{KolchinBook1973}. For such sets we use the
notation $A = A_1,\ldots,A_p$ to specify the list of the elements
of $A$ arranged in order of increasing rank.
We denote the sets of initials and separants of
elements of $A$ by $i_A$ and $s_A$,
 respectively. Let $H_A=i_A\cup s_A$.
For a finite set $S$  in a commutative ring $R$, denote 
 the smallest multiplicative set containing $1$ and
 $S$ by $S^\infty$. Let $I$ be an ideal of $R$.
The {\it colon ideal} $I:S^\infty$ is defined as
$$\{a \in R\:|\:\exists s \in S^\infty: sa \in I\}.$$ If $R$ is a differential ring $I$ is a
differential ideal, then $I:S^\infty$ is also a differential ideal
(see \cite{KolchinBook1973}).

Let $A = A_1,\ldots,A_r$ and $B = B_1,\ldots,B_s$
 be autoreduced  sets. We say the rank of $A$ is lower than the rank of
$B$ if
\begin{itemize}
\item
there exists $k \leq \min(r, s)$ such that $\rank A_i$ =
$\rank B_i$ for all $i$, $1 \leq i < k$, and $\rank A_k < \rank B_k$,
\item or if $r > s$ and $\rank A_i = \rank B_i$ for all $i$, $1 \leq i \leq s$.
\end{itemize}
We say that
 $\rank A = \rank B$ if $r=s$ and
$\rank A_i = \rank B_i$ for all $i$, $1 \leq i \leq r$. 
For instance, let $A := y_1'+y_2, y_2'+y_1$ and $B := y_1'+y_2$. Then, in an orderly ranking with $y_1 < y_2$, the rank of $A$ is lower than the rank of $B$ because $\rank A_1=\rank B_1$, but $A$ has more elements than $B$ does. For an orderly ranking with $y_2 < y_1$, we write $A = y_2'+y_1,y_1'+y_2$, and the rank of $A$ is lower than the rank of $B$ because now $\rank A_1 <\rank B_1$.

For a differential ideal $I \subset K\{y_1,\ldots,y_n\}$, its autoreduced subset of the
least rank is called a {\it characteristic set} of $I$ \cite[p.~82]{KolchinBook1973}. If the differential ideal $I$ is prime and $C$ is its characteristic set, then $I = [C]:H_C^\infty$ \cite[Lemma~2, Chapter~IV, Section~9]{KolchinBook1973}. For example, the smallest radical differential ideal $I$ of $K\{y\}$ containing $y'^2+y$ is not prime, having two minimal differential prime ideals containing it, $P_1 = [y'^2+y, 2y''+1]$ and $P_2=[y]$. Moreover, $y'^2+y$ and $y$ are characteristic sets of $P_1$ and $P_2$, respectively, and $P_1 = [y'^2+y]:y'^\infty$. Also, $y'^2+y$ is a characteristic set of $I$, but $[y'^2+y]:y'^\infty$ strictly contains $I$.

\begin{theorem}
\label{theorem:fgtrans}
Consider any computable ordinary nonconstant differential field $K$, and
assume that $\Khat$ is not an algebraic field extension of $K$.
Let $K\la z\ra$ be
a computable differential field extension generated by an element
$z$ which is differentially transcendental over $K$, presented so that
$K$ is a computably enumerable subset of $K\la z\ra$.
Then the constraint set $\overline{T_{K\la z\ra}}$
is computable in an oracle for $T_K$.  Also, the constrainability set
$\overline{U_{K\la z\ra}}$ is computably enumerable relative
to $T_K$.  Both the computation of $T_{K\la z\ra}$
and the enumeration of $\overline{U_{K\la z\ra}}$
are uniform in $z$ and $T_K$.
\end{theorem}
\begin{proof}
It is only necessary to show that $T_{K\la z\ra}$
has both an existential definition and a universal definition,
in which the quantifier-free parts are allowed to
use the relation of membership in the oracle set $T_K$.
Of course, $T_{K\la z\ra}$ is computably enumerable
without any oracle; indeed, its existential definition
comes straight from Definition \ref{definition:constraint},
with no $T_K$-oracle required:
\begin{align*}
(p,q)\in T_{K\la z\ra} \iff &\big(\exists x,y\in\widehat{K\la z\ra}\big)(\exists h\in K\la z\ra\{ Y\})\\
&[h(x)=0=p(x)\neq q(x)~\&~h(y)\neq 0=p(y)\neq q(y)].
\end{align*}
So we need only give a universal definition of $T_{K\la z\ra}$
(equivalently, an existential definition of $\overline{T_{K\la z\ra}}$)
relative to $T_K$.  For any $p,q\in K\la z\ra\{ Y\}$, clear the denominators
of the coefficients to form polynomials $p_0,q_0\in K\{ z,Y\}$.
We claim that
\begin{align*}
\label{eq:Edefinition}
(p,q)\in &\overline{T_{K\la z\ra}} \iff
(\exists f,g\in K\{ Z\})\big(\exists \ztilde \in\Khat\big) \big[(f,g)\in\TKbar~\&~f(\ztilde)=0\neq g(\ztilde)\\
&\&~(p_0(\ztilde,Y),q_0(\ztilde,Y))\in\overline{T_{K\la\ztilde\ra}}~\&~p\in
K\big(z,\del z,\ldots,\del^{(\ord(f)-\ord_Y(p)-1)/2}z\big)\{ Y\}\big]
\end{align*}
Once we have proven this equivalence, we will have
a decision procedure for $T_{K\la z\ra}$ relative to $T_K$:
given input $(p,q)$, search simultaneously for witnesses
to either of the two existential statements above.
(Notice that a $T_K$-oracle will decide membership in $T_{K\la\ztilde\ra}$,
for every $\ztilde\in\Khat$, uniformly in $\ztilde$, according
to Theorem \ref{theorem:fgalg}.)  Eventually,
this procedure must find a witness for one or the other,
and when it does, we have determined whether $(p,q)\in T_{K\la z\ra}$ or not.

For the forwards implication, let $(p,q)\in\overline{T_{K\la z\ra}}$.
Then, for every $h_0\in K\{ z,Y\}$ of rank (with respect to $Y$) less
than the rank of $p_0$, we know that $$1\in \sqrt{[p_0,h_0]}:q_0,$$
by Corollary \ref{cor:constraintvariety}.
Therefore, no matter what $f$ and $g$ we choose from $K\{ Z\}$,
we will have $$1\in \sqrt{[p_0,h_0,f]}:(g\cdot q_0)$$ for every such $h_0$,
and hence $1$ will also lie in the ideal $$\sqrt{[p_0(\ztilde,Y),h_0(\ztilde,Y)]}:q_0(\ztilde,Y)\subset K\la\ztilde\ra\{ Y\}$$ for each $\ztilde\in\Khat$ satisfying $(f,g)$.
Set $$\ptilde(Y)=p_0(\ztilde,Y)\quad \text{and}\quad \qtilde(Y)=q_0(\ztilde,Y).$$
Now every $\htilde\in K\la\ztilde\ra\{ Y\}$ of rank (with respect to $Y$)
less than the rank of $\ptilde$ has lower rank than $p$,
since $\ptilde$ has rank $\leq$ the rank of $p$,
and so viewing $\htilde$ as a polynomial in both
$\ztilde$ and $Y$ and clearing denominators yields an $h_0$ as above
and shows that $$1\in\sqrt{\left[\ptilde,\htilde\right]}:\qtilde.$$
Then we apply Corollary \ref{cor:constraintvariety} once more to see
that $(\ptilde,\qtilde)$ lies in $\overline{T_{K\la z\ra}}$, as required.
By Lemma~\ref{lem:difoveralgclosed}, there must exist a pair $(f,g)\in\TKbar$ with $\ord_z(f)$
large enough that $$p\in K\big(z,\del z,\ldots,\del^{(\ord_z(f)-\ord_Y(p)-1)/2}z\big)\{ Y\},$$
and this pair, along with any $\ztilde\in\Khat$ satisfying it,
satisfies the existential condition given.

We prove the backwards implication by contraposition.
Suppose that $(p,q)\in T_{K\la z\ra}$.
Fix an $r$ such that $$p,q\in K\big(z,\del z,\ldots,\del^rz\big)\{ Y\}.$$
Then there exist elements $x$ and $y$ in the
differential closure of $K\la z\ra$ which both satisfy
$(p,q)$, but such that some $h\in K\{ z,Y\}$
has
\begin{equation}\label{eq:hzxy}
h(z,x)=0\neq h(z,y).
\end{equation}
 Let $h_0(z,Y)\in K\{ z,Y\}$ be the result of multiplying
$h$ by the least common multiple of all denominators
of its coefficients, as with $p_0$ and $q_0$.
 By Lemma \ref{lemma:counterEX},
we may assume that $h_0$  is of rank lower than that of $p_0$
(under the orderly ranking with $z < Y$), that is,
\begin{equation}\label{eq:h0}
\ord_z(h_0) +\ord_Y(h_0) \le \ord_z(p_0) +\ord_Y(p_0) =: b
\end{equation}
holds.
It follows from~\eqref{eq:hzxy} that
\begin{equation}\label{eq:1notin}
1\notin \sqrt{[p,h]}:q\quad\text{and}\quad h\notin\sqrt{[p]}:q
\end{equation}
implying that
\begin{equation}\label{eq:capz}
\sqrt{[p_0,h_0]}:q_0\cap K\{z\} = \{0\}.
\end{equation}
It follows from~\eqref{eq:capz} that there exists  a minimal differential prime component of $$\sqrt{[p_0,h_0]}:q_0$$ such that one of its characteristic sets $C$ with respect to the elimination  ranking with $z < Y$ has the form
$\{C_1\}$ for some irreducible $C_1 \in K\{z,Y\} \setminus K\{z\}$. Indeed, suppose it had  two elements in it, $B_1 < B_2$. We know that a characteristic set is autoreduced, and that $z < Y$ and the ranking is elimination.  So, the leaders of  $B_1$ and $B_2$ must be derivatives of $z$ and $Y$, respectively. Again, since the ranking is elimination and $z < Y$, $B_1$ cannot depend on $Y$, which would contradict~\eqref{eq:capz}.
It now follows from \cite[Proposition~14]{Bounds} and~\eqref{eq:h0} that
\begin{align*}&\ord_z (C_1) \le \ord_z (C_1) + \ord_Y(C_1) \le\\
&\le \max(\ord_z(h_0),\ord_z(q_0),\ord_z(p_0)) + \max(\ord_Y(h_0),\ord_Y(q_0),\ord_Y(p_0)) \le 2b.
\end{align*}
Moreover, for every irreducible
$C_0\in K\{z\}$ with $\ord_z (C_0) > 2b$, the set $C := \{C_0,C_1\}$ is autoreduced and,
by the Rosenfeld Lemma (see \cite[Theorem~4.8]{HubertDif}),
\begin{equation}\label{eq:1notinf}
1\in [C]:H_C^\infty\iff 1 \in (C):H_C^\infty
\end{equation} with the latter statement being impossible because of the choice of $C_0$ and $C_1$ (see \cite[pages~88--90]{Ritt}).
Fix any $(f,g)\in\TKbar$ with $f$ of order $\ge 2b+1$, and any $\ztilde\in \Khat$ satisfying $(f,g)$,
and let $\tilde p$, $\tilde q$, and $\tilde h$
be the result of replacing $z$ by $\tilde z$ in $p_0$, $q_0$, and $h_0$,
respectively. 
Now~\eqref{eq:1notin} and~\eqref{eq:1notinf} imply that
\begin{equation}\label{eq:1notintilde}
1 \notin \sqrt{{\left[\tilde p,\tilde h\right]}}: \tilde q \subset \Khat\{Y\}.
\end{equation}
Therefore, Corollary \ref{cor:constraintvariety} shows that $(\ptilde,\qtilde)\in T_{K\la\ztilde\ra}$.
This completes the proof of the backwards direction, and shows
that our existential definition of $\overline{T_{K\la z\ra}}$ was correct.
\end{proof}
\setlength{\bibsep}{0.0pt}
\bibliographystyle{model1b-num-names}
\normalsize
\bibliography{mainAlexey}
\end{document}